\documentclass[12pt]{article}
\usepackage{amssymb,amsmath,bm}
\usepackage[colorlinks=true, urlcolor=blue,linkcolor=blue, citecolor=blue]{hyperref}
\usepackage{mathrsfs,verbatim}
\usepackage{caption,cleveref }
\usepackage{graphicx, float}
\usepackage{color, mathtools, tikz, xcolor}
\usepackage{enumerate}
\usepackage{tikz}
\usetikzlibrary{shapes.geometric}
\usepackage{pdfpages}

\usepackage{caption}

\usetikzlibrary{positioning}
\usetikzlibrary{arrows}

\usepackage{subcaption}

\usetikzlibrary{calc}

\usetikzlibrary{decorations.pathreplacing}

\usepackage{appendix}
\usepackage{array}

\usepackage{multirow}
\usepackage{graphicx}
\newcolumntype{L}{>{$}l<{$}}
\newcolumntype{D}{>{\centering\arraybackslash}p{3.5cm}}
\newcolumntype{C}{>{$}c<{$}}

\begin{document}

\newcommand{\bea}{\begin{eqnarray}}
\newcommand{\ena}{\end{eqnarray}}
\newcommand{\beas}{\begin{eqnarray*}}
\newcommand{\enas}{\end{eqnarray*}}
\newcommand{\beq}{\begin{equation}}
\newcommand{\enq}{\end{equation}}
\def\qed{\hfill \mbox{\rule{0.5em}{0.5em}}}
\newcommand{\bbox}{\hfill $\Box$}
\newcommand{\ignore}[1]{}
\newcommand{\ignorex}[1]{#1}
\newcommand{\wtilde}[1]{\widetilde{#1}}
\newcommand{\mq}[1]{\mbox{#1}\quad}
\newcommand{\bs}[1]{\boldsymbol{#1}}
\newcommand{\qmq}[1]{\quad\mbox{#1}\quad}
\newcommand{\qm}[1]{\quad\mbox{#1}}
\newcommand{\nn}{\nonumber}
\newcommand{\Bvert}{\left\vert\vphantom{\frac{1}{1}}\right.}
\newcommand{\To}{\rightarrow}
\newcommand{\supp}{\mbox{supp}}
\newcommand{\law}{{\cal L}}
\newcommand{\Z}{\mathbb{Z}}
\newcommand{\mc}{\mathcal}
\newcommand{\mbf}{\mathbf}
\newcommand{\tbf}{\textbf}
\newcommand{\lp}{\left(}
\newcommand{\limm}{\lim\limits}
\newcommand{\limminf}{\liminf\limits}
\newcommand{\limmsup}{\limsup\limits}
\newcommand{\rp}{\right)}
\newcommand{\mbb}{\mathbb}
\newcommand{\rainf}{\rightarrow \infty}
\newtheorem{problem}{Problem}[section]
\newtheorem{exercise}{Exercise}[section]
\newtheorem{theorem}{Theorem}[section]
\newtheorem{corollary}{Corollary}[section]
\newtheorem{conjecture}{Conjecture}[section]
\newtheorem{proposition}{Proposition}[section]
\newtheorem{lemma}{Lemma}[section]
\newtheorem{definition}{Definition}[section]
\newtheorem{example}{Example}[section]
\newtheorem{remark}{Remark}[section]
\newtheorem{solution}{Solution}[section]
\newtheorem{case}{Case}[section]
\newtheorem{condition}{Condition}[section]
\newtheorem{assumption}{Assumption}
\newtheorem{defn}{Definition}[section]
\newtheorem{eg}{Example}[section]
\newtheorem{thm}{Theorem}[section]
\newtheorem{lem}{Lemma}[section]
\newtheorem{soln}{Solution}[section]
\newtheorem{propn}{Proposition}[section]
\newtheorem{ex}{Exercise}[section]
\newtheorem{conj}{Conjecture}[section]
\newtheorem{pb}{Problem}[section]
\newtheorem{cor}{Corollary}[section]
\newtheorem{rmk}{Remark}[section]
\newtheorem{note}{Note}[section]
\newtheorem{notes}{Notes}[section]
\newtheorem{readingex}{Reading exercise}[section]
\newcommand{\pf}{\noindent {\bf Proof:} }
\newcommand{\proof}{\noindent {\it Proof:} }
\frenchspacing

\tikzstyle{level 1}=[level distance=2.75cm, sibling distance=5.65cm]
\tikzstyle{level 2}=[level distance=3cm, sibling distance=2.75cm]
\tikzstyle{level 3}=[level distance=3.9cm, sibling distance=1.5cm]

\tikzstyle{bag} = [text width=10em, text centered] 
\tikzstyle{end} = [circle, minimum width=3pt,fill, inner sep=0pt]

\title{Defective Ramsey Numbers in Graph Classes \thanks{Supported by T{\"U}B\.{I}TAK Grant no:118F397.}}
\author{Yunus Emre Demirci\footnote{Department of Mathematics, Bo\u{g}azi\c{c}i University, 34342, Bebek, Istanbul, Turkey} \hspace{0.2in} T{\i}naz Ekim \thanks{This author acknowledges the support of the Turkish Academy of Science T{\"U}BA GEBIP award.} \footnote{Department of Industrial Engineering, Bo\u{g}azi\c{c}i University, 34342, Bebek, Istanbul, Turkey} \hspace{0.2in} John Gimbel\thanks{This work was conducted when John Gimbel visited Istanbul Center for Mathematical Sciences (IMBM) whose support is greatly acknowledged.} \footnote{Mathematics and Statistics, University of Alaska, 34342, Fairbanks, AK, 99775-6660, USA} \hspace{0.2in} Mehmet Akif Y{\i}ld{\i}z\footnote{Department of Mathematics, Bo\u{g}azi\c{c}i University, 34342, Bebek, Istanbul, Turkey} } \vspace{0.25in}

\maketitle
 
\begin{abstract}
\noindent Given a graph $G$, a \textit{$k$-sparse $j$-set} is a set of $j$ vertices inducing a subgraph with maximum degree at most $k$. A \textit{$k$-dense $i$-set} is a set of $i$ vertices that is $k$-sparse in the complement of $G$. As a generalization of Ramsey numbers, the $k$-defective Ramsey number $R_k^{\mathcal{G}}(i,j)$ for the graph class $\mathcal{G}$ is defined as the smallest natural number $n$ such that all graphs on $n$ vertices in the class $\mathcal{G}$ have either a $k$-dense $i$-set or a $k$-sparse $j$-set. In this paper, we examine $R_k^{\mathcal{G}}(i,j)$ where $\mathcal{G}$ represents various graph classes. In forests and cographs, we give formulas for all defective Ramsey numbers. In cacti, bipartite graphs and split graphs, we provide defective Ramsey numbers in most of the cases and point out open questions, formulated as conjectures if possible. 

\bigskip


\end{abstract}

\section{Introduction}

Ramsey numbers have been the focus of several research papers since decades. For any two positive integers $i$ and $j$, the Ramsey number $R(i, j)$ is the smallest positive integer such that every graph on at least $R(i, j)$ vertices has a clique of size $i$ or an independent set of size $j$. It is well-known that computing Ramsey numbers is an extremely difficult task starting from quite small integers $i$ and $j$. Among several variations of the classical Ramsey numbers, some research has been focused on defective Ramsey numbers recently. This version is obtained by relaxing the notion of cliques and independent sets as follows. A \textit{$k$-sparse $j$-set} is a set $S$ of $j$ vertices of a graph $G$ such that the subgraph induced by $S$ has maximum degree at most $k$. A \textit{$k$-dense $i$-set} is a set $D$ of $i$ vertices of a graph $G$ that is $k$-sparse in the complement of $G$. In this case, we also say that each vertex in $D$ \textit{misses} at most $k$ other vertices in $D$. A $k$-sparse or $k$-dense set is called \textit{$k$-defective} (or \textit{$k$-uniform}).  The \textit{$k$-defective Ramsey number} $R_k^{\mathcal{G}}(i,j)$ for the graph class $\mathcal{G}$ is defined as the smallest natural number $n$ such that all graphs on $n$ vertices in the class $\mathcal{G}$ have either a $k$-dense $i$-set or a $k$-sparse $j$-set.

Exacts values of some 1-defective and 2-defective Ramsey numbers are reported by Cockayne and Mynhardt (under the name of dependent Ramsey numbers) in \cite{dependentRamsey} and by Ekim and Gimbel in \cite{defectiveparameterTinazJohn}. More recently, in addition to direct proofs, some computer based search algorithms are also used by Akdemir and Ekim \cite{defectiveparameterTinazAhu}, and by Chappell and Gimbel \cite{defectiveRamseyJohnChappell} to obtain further 1-defective and 2-defective Ramsey numbers.  To deal with hard problems, a natural approach in graph theory consists in considering the same problem when restricted to special graph classes. This method was adopted in a systematic way by Belmonte at al. in \cite{ramseygraphclassesHeggernes} for computing (classical) Ramsey numbers in various graph classes. In the same paper, related results in the literature are also surveyed. The approach of considering Ramsey numbers in graph classes has been recently applied to the defective version by Ekim et al.  in \cite{1defectiveperfectTinazOylumJohn}. The authors compute some exact values of 1-defective Ramsey numbers in the class of perfect graphs. Namely, they show $R^{\mathcal{PG}}_1 (3, j) = j$ for any $j \geq 2$, $R^{\mathcal{PG}}_1 (4, 4) = 6$, $R^{\mathcal{PG}}_1 (4, 5) = 8$, $R^{\mathcal{PG}}_1 (4, 6) = 10$, $R^{\mathcal{PG}}_1 (4, 7) = 13$, $R^{\mathcal{PG}}_1 (4, 8) = 15$ and $R^{\mathcal{PG}}_1 (5, 5) = 13$  where $\mathcal{PG}$ denotes the class of perfect graphs. Among further research directions, the computation of defective Ramsey numbers in more restricted graph classes where a formula describing all defective Ramsey numbers can be derived is pointed out as a promising direction. The classes of cographs, interval graphs and cacti are explicitly mentioned as possible candidates.

\textbf{Our contribution:} In this paper, we consider defective Ramsey numbers in various graph classes, namely, forests, cacti, bipartite graphs, split graphs and cographs. In forests and cographs we compute all defective Ramsey numbers, formulated as $j+\Big\lfloor \dfrac{j-1}{k+1}\Big\rfloor$ and $1+\dfrac{(i-1)(j-1)-\{i-1\}\{j-1\}}{k+1}$ where $\{x\}$ denotes the value of the integer $x$ modulo $k+1$, in Theorems \ref{thm:forest} and \ref{thm:cographkdefective}, respectively. In cacti, the formula $j-1+\Big\lceil\dfrac{j-1}{k}\Big\rceil$ gives all defective Ramsey numbers except a few cases that we point out as open question. In bipartite graphs, we show that all 1-defective Ramsey numbers are $2j-1$ (Theorems \ref{thm:bipartiteexcept4} and \ref{thm:bipartite1defective}) with a few exceptions for small values of $j$ (Theorem \ref{thm:bipartitesmall}) and five open cases for which we conjecture also $2j-1$ (Conjecture \ref{con:bipartite1defectivtherest}). In addition, we settle all defective Ramsey numbers in bipartite graphs for $k\geq 2$ and $i\geq 2k+3$ in Theorem \ref{thm:bipartitelargevalues} as follows: if $j\geq 2k+1$ then it is $2j-1$, otherwise it is $2j-1-k$.  As for split graphs, we provide all 1-defective and 2-defective Ramsey numbers (Theorems \ref{thm:splitsufficientlylarge} and  \ref{thm:split1and2defective} respectively). We also show in Theorem \ref{thm:splitsufficientlylarge} that defective Ramsey numbers in split graphs are equal to $i+j-1$ for all  ``relatively large" $i$ and $j$. We conclude with a conjecture for a general formula in split graphs covering all known cases (Conjecture \ref{conj:splitmostgeneral}).

Last but not least, in Section \ref{sec:conj}, we consider a conjecture formulated by Chappell and Gimbel in \cite{defectiveRamseyJohnChappell} which states that $R_k(k+i,k+j)-k\leq R(i,j)$ holds for all $i,j,k\geq 0$ (for general graphs). In light of our results, we show that this conjecture holds when restricted to forests, cacti or cographs, whereas it fails when restricted to bipartite graphs or split graphs. In Section \ref{sec:conclusion}, we summarize our results, conjectures and open questions in Table \ref{table:summary}, and point out some research directions.

\section{Definitions and Preliminaries}

Given a graph $G=(V,E)$, for a vertex $u\in V$, $N(u)$ denotes the set of neighbors of $u$. The \textit{degree} of a vertex $u$ is the number of its neighbors, denoted by $d(u)$. We also have $N[u]=N(u) \cup \{u\}$. For a subset of vertices $S\subset V$, the number of neighbors of a vertex $u$ in $S$ is denoted by $d_S(u)$. Whenever it is clear from the context (or we mean the whole vertex set), we omit the subscript $S$. Given two graphs $G_1=(V_1,E_1)$ and $G_2=(V_2,E_2)$, the \textit{disjoint union} of $G_1$ and $G_2$ is the graph with vertex set $V_1\cup V_2$ and edge set $E_1\cup E_2$. We also denote by $kG$ the disjoint union of $k$ copies of a graph $G$. The \textit{join} of $G_1$ and $G_2$ is the graph $G=(V_1\cup V_2, E_1\cup E_2\cup \{xy, x\in V_1, y\in V_2\})$. For two subsets of vertices $U$ and $V$, we say that $U$ is \textit{complete} to $V$ (or vice versa) if there is an edge between every pair of vertices $u\in U$ and $v\in V$. A graph class $\mathcal G$ is said to be \textit{self-complementary} if for any graph in $\mathcal G$, its complement is also in $\mathcal G$.

A \textit{clique} is a set of vertices which are pairwise adjacent. An \textit{independent set} is a set of vertices which are pairwise non-adjacent. Given a graph $G$, the size of a largest $k$-sparse set in $G$ is denoted by $\alpha_k(G)$. A graph $G$ is \textit{bipartite} if its vertex set can be partitioned into two independent sets $A$ and $B$; then $(A,B)$ is called a bipartition of $G$. We use the notation $K_{i,j}$ for a complete bipartite graph with bipartition $(A,B)$ such that $|A|=i$ and $|B|=j$. A connected component $C$ of a graph is a maximally connected subset of vertices of $G$. We sometimes abuse the language and use $C$ and $G-C$  to denote the subgraph of $G$ induced by $C$ and by $V(G)\setminus C$ respectively.  

Let us first state some remarks that will be frequently used. The followings are either trivial or a direct consequence of some results from \cite{1defectiveperfectTinazOylumJohn} and they will be sometimes used without reference throughout the paper.

\begin{remark}\label{rem:generalsmallvaluesofiandj}
For all $i,j$ such that  $\min\{i,j\}\leq k+1$, we have $R_{k}^{\mathcal{G}}(i,j)=\min\{i,j\}$.
\end{remark}

\begin{remark}\label{rem:graphclasshereditary}
Let $\mathcal{G}$ and $\mathcal{H}$ be two graph classes with $\mathcal{G}\subseteq\mathcal{H}$. Then, $R_{k}^{\mathcal{G}}(i,j)\leq R_{k}^{\mathcal{H}}(i,j)$ for all $i,j,k$.
\end{remark}

\begin{remark}\label{rem:defectiveinequalityfork}
For any graph class $\mathcal{G}$, we have $R_{k+1}^{\mathcal{G}}(i,j)\leq R_{k}^{\mathcal{G}}(i,j)$ for all $i,j,k$.
\end{remark}

\begin{remark}\label{rem:defectiveinequalityforiandj}
For any graph class $\mathcal{G}$, we have $R_{k}^{\mathcal{G}}(i,j)\leq R_{k}^{\mathcal{G}}(i+a,j+b)$ for all $i,j,k,a,b$.
\end{remark}

\begin{remark}\label{rem:defectivecomplement}
Let $\mathcal{G}$ be a self-complementary graph class. Then, $R_{k}^{\mathcal{G}}(i,j)=R_{k}^{\mathcal{G}}(j,i)$ for all $i,j,k$.
\end{remark}

\begin{remark}\label{rem:sparseunion}
The disjoint union of $k$-sparse sets is also a $k$-sparse set.
\end{remark}

\begin{remark}\label{rem:ksparseunion}
Let $C_1$, $C_2$, ..., $C_t$ be connected components of a graph $G$. Then, $$\alpha_k(G)=\sum_{l=1}^{t}\alpha_k(C_l)$$ 
\end{remark}

\noindent In \cite{1defectiveperfectTinazOylumJohn}, the authors noted $R_{1}^{\mathcal{PG}}(3,j)=j$ for all $j\geq 3$. This observation can be generalized as follows.

\begin{lemma}\label{lem:generalkplustwo}
Let $\mathcal{G}$ be a graph class containing all empty graphs. Then, $$R_{k}^{\mathcal{G}}(k+2,j)=j\text{ for all }j\geq k+2$$
\end{lemma}

\begin{proof}
Firstly, if we take an empty graph on $j-1$ vertices, clearly it has no $k$-dense $(k+2)$-set and no $k$-sparse $j$-set. Secondly, let $G$ be a graph on $j$ vertices and assume it is not $k$-sparse. Hence there exists a vertex $u$ with $d(u)\geq k+1$. Choose exactly $k+1$ neighbors of $u$, say $v_1$, $v_2$, ..., $v_{k+1}$. Then, $\{u,v_1,v_2,...,v_{k+1}\}$ forms a $k$-dense $(k+2)$-set. \qed \\
\end{proof}

\noindent Since all graph classes under consideration in this paper contain all empty graphs, as a natural consequence of Remark \ref{rem:generalsmallvaluesofiandj} and Lemma \ref{lem:generalkplustwo}, our main focus will be the computation of defective Ramsey numbers where $i\geq k+3$ and $j\geq k+2$.


\section{Forests}
A \textit{forest} is a graph whose connected components are trees. Let $\mathcal{FO}$ be the class of all forests. In this section, we give a lower bound on the maximum size of a $k$-sparse set in a forest, namely Lemma \ref{lem:forestksparse}, which will be useful in the computation of defective Ramsey numbers in cacti in Section \ref{sec:cacti}. The same lower bound is also used to derive the formula for all defective Ramsey numbers in forests in Theorem \ref{thm:forest}.\\

\noindent Let us first state the classical Ramsey numbers in forests:

\begin{theorem}\label{thm:forestgeneralRamsey}
\cite{ramseygraphclassesHeggernes} For all $i\geq 3$ and $j\geq1$, we have $R_0^{\mathcal{FO}}(i,j)=2j-1$.
\end{theorem}

\noindent Since an empty graph is a forest, the following remark follows from Lemma \ref{lem:generalkplustwo}.

\begin{remark}\label{rem:forestkplustwo}	For all $k\geq1$ and $j\geq k+2$, we have $R_k^{\mathcal{FO}}(k+2,j)=j$.
\end{remark}

\noindent The following emphasizes that a forest does not have large $k$-dense sets.

\begin{lemma}\label{lem:forestkdense}
	Let $k\geq1$, $i\geq k+3$ and $G$ be a forest. Then, $G$ has no $k$-dense $i$-set.
\end{lemma}

\begin{proof}
Let $n$ be the order of $G$. If $n\leq i-1$, then the claim is trivial. Assume $n\geq i$ and consider a subset of vertices $D$ of order $i$. Since $G$ is a forest, $D$ does not induce any cycle. Hence it has a vertex of degree at most one, which misses at least $i-2$ other vertices in $D$. Since $i-2\geq k+1$, $D$ is not a $k$-dense $i$-set. \qed\\
\end{proof}

\noindent In contrast with $k$-dense sets, forests admit  ``relatively large" $k$-sparse sets as shown in the following:
\begin{lemma}\label{lem:forestksparse}
	If $F$ is a forest of order $n$ then $\alpha_k(F)\geq \Big\lceil\dfrac{k+1}{k+2}n\Big\rceil$.
\end{lemma}

\begin{proof}
	Fix $k$. We will proceed by induction on $n$. If $n\leq k+2$ and $F$ has no vertex of degree $k+1$ then $F$ is $k$-sparse. If $F$ has a vertex of degree $k+1$ then $F$ is a star and all vertices except the center are independent and thus $k$-sparse, completing the base case. So suppose $n\geq k+3$ and the lemma holds for forests of smaller order.\\
	
	\noindent Let us say a vertex is {\em large} if it has degree at least $k+1$. If $F$ contains no large vertices then it is $k$-sparse and we are done. So suppose that $F$ has a component $C$ with exactly one large vertex $w$.
	Note, $C-w$ is $k$-sparse. In which case say $C$ has $m$ vertices. Note $m\geq k+2$. Perform induction and note $F-C$ has a $k$-sparse set of order at least $\Big\lceil\dfrac{k+1}{k+2}(n-m)\Big\rceil$. Adding to this set the vertices of $C-w$ produces the desired result.\\
	
	\noindent So suppose $F$ has a component $C$ with more than one large vertex. Find two, say $u$ and $v$, that are furthest apart. Note, $C-v$ has a component containing all large vertices of $C$ other than $v$, call
	the component $H$. Say $H$ has order $h$. Note, $v$ is adjacent with at least $k$ vertices not in $H$. Suppose $v$ has degree at least $k+2$. Perform induction to produce a $k$-sparse subset of $H$ having at least $\Big\lceil\dfrac{k+1}{k+2}h\Big\rceil$ vertices. Add to this all vertices of $C-(V(H)\cup \{v\})$ and reach the desired conclusion. Finally, suppose $v$ has degree exactly $k+1$. Let $x$ be the vertex in $H$ adjacent with $v$. In $C-x$ we find a component containing $v$. Let us call this $D$. This component contains at least $k+1$ vertices and is $k$-sparse. Removing this component from $F$, along with $x$, we perform
	induction on what remains and then add all vertices of $D$ and reach our desired conclusion. \qed \\
\end{proof}

\noindent Now, we will present the exact formula for defective Ramsey numbers on forests.
\begin{theorem}\label{thm:forest}
	For all $k\geq1$, $i\geq k+3$ and $j\geq k+2$, we have $$R_k^{\mathcal{FO}}(i,j)=j+\Big\lfloor\dfrac{j-1}{k+1}\Big\rfloor.$$
\end{theorem}

\begin{proof}
Let $j-1=(k+1)s+t$ for some $s\geq 1$ and $0\leq t\leq k$. We need to show that $R_k^{\mathcal{FO}}(i,j)=(k+1)s+t+1+\Big\lfloor\dfrac{(k+1)s+t}{k+1}\Big\rfloor=(k+2)s+(t+1)$.\\
	
\noindent Let $F$ be a forest of order $(k+2)s+(t+1)$. By Lemma \ref{lem:forestksparse}, we have $$\alpha_k(F)\geq \Big\lceil\dfrac{((k+2)s+(t+1))(k+1)}{k+2}\Big\rceil=(k+1)s+\Big\lceil t+1-\dfrac{t+1}{k+2}\Big\rceil=(k+1)s+(t+1)=j$$ since $1\leq t+1\leq k+1$. Then, any forest on $j+\Big\lfloor\dfrac{j-1}{k+1}\Big\rfloor$ vertices has a $k$-sparse $j$-set.\\

\noindent Now, consider the graph $H$ which consists of the disjoint union of $sK_{1,k+1}$ and $t$ isolated vertices. Clearly $H$ has no cycles, thus it is a forest on $(k+2)s+(t+1)-1$ vertices. Also we have $\alpha_k(H)\leq(k+1)s+t=j-1$ since any $k$-sparse set misses at least one vertex from each star. Moreover, since $i\geq k+3$, $H$ has no $k$-dense $i$-set by Lemma \ref{lem:forestkdense}.  This completes the proof. \qed
\end{proof}


\section{Cacti}\label{sec:cacti}

In this section we discuss \textit{cacti}, those graphs where each block is a cycle, a single edge or a single vertex. Some texts define cacti in this way but further restrict them to being connected. If the reader insists on this definition, there will be no disappointment as each of the following results for cacti still hold given the requirement of connectivity. At first glance, it might appear the defective Ramsey numbers for cacti are trivial as relatively large dense sets aren't found in cacti.
In fact, no cactus contains $k+4$ vertices that are $k$-dense. This is because every subset of $k+4$ vertices in a cactus induces again a cactus and each cactus have minimum degree at most two. We also note that in any cactus no cycle contains a chord. Hence, every open neighborhood of every vertex in a cactus is $1$-sparse.\\

\noindent Let $\mathcal{CA}$ denote the set of all cacti. Let us say an \textit{end} block is a block which has exactly one cut vertex. An \textit{isolated} cycle is a cycle which has no cut vertex. A \textit{pendant} vertex is a vertex of degree exactly one. A \textit{penultimate} vertex is a vertex adjacent with a pendant vertex and at most one nonpendant vertex. We can now state the following.

\begin{theorem}\label{thm:cactusgeneralRamsey}
\cite{ramseygraphclassesHeggernes} For all $i\geq 3$ and $j\geq 1$ we have $R_{0}^{\mathcal{CA}}(i,j)=\begin{cases}
\Big\lfloor\dfrac{5(j-1)}{2}\Big\rfloor + 1, & \text{ if } i=3 \\
3(j-1)+1, & \text{ if } i\geq4
\end{cases}$
\end{theorem}

\noindent We will extend this momentarily. But first, we make a few observations.\\

\noindent In any graph every $j$-set is $k$-sparse and $k$-dense provided $j\leq k+1$ due to Remark \ref{rem:generalsmallvaluesofiandj}. Hence, in the cases where $i$ or $j$ are at most $k+1$ we have $R_{k}^{\mathcal{CA}}(i,j)=\min\{i,j\}$. So in the following when noting Ramsey numbers of the sort $R_{k}^{\mathcal{CA}}(i,j)$ let us assume $i$ and $j$ are both at least $k+2$.

\begin{remark}\label{rem:cactuskplustwo}
For all $k\geq 1$ and $j\geq k+2$, we have $R_{k}^{\mathcal{CA}}(k+2,j)=j$.
\end{remark}

\begin{proof}
Since the empty graph is a cactus graph, the result follows from Lemma \ref{lem:generalkplustwo}. \qed
\end{proof}

\begin{lemma}\label{lem:cactuskdense}
Let $k\geq1$, $i\geq k+4$. Then, a cactus $G$ has no $k$-dense $i$-set.
\end{lemma}

\begin{proof}
Let $G$ be a cactus of order $n$. If $n\leq i-1$, then the claim is trivial. Assume $n\geq i$ and consider a subset of vertices $D$ of order $i$. Since $G[D]$ is a cactus, there is a vertex $v\in D$ such that $d_D(v)\leq 2$. Thus, $v$ misses at least $i-3$ other vertices in $D$. Since $i-3\geq k+1$, $D$ is not a $k$-dense $i$-set. \qed\\
\end{proof}

\noindent As no cactus contains a $k$-dense set of order $k+4$, we note that determination of $R_{k}^{\mathcal{CA}}(i,j)$ in the case where $i\geq k+4$ amounts to determining the smallest $n$ with the property that every cactus of order $n$ contains a $k$-sparse set of order $j$. First, we will start with the examination of $1$-defective case.

\begin{lemma}\label{lem:cactus1sparse}
If $G$ is a cactus of order $n$ then $\alpha_1(G)\geq \Big\lfloor\dfrac{n}{2}\Big\rfloor+r(n)$ where $$r(n)=\begin{cases}
0, & \text{ if } n\equiv 0\pmod{4}\\
1, & \text{ otherwise } 
\end{cases}$$
\end{lemma}

\begin{proof}
We will proceed by induction on $n$. The statement is clearly true if $n\leq4$. So suppose the lemma holds for all cacti on less than $n\geq5$ vertices. Let $G$ be a cactus with exactly $n$ vertices. Suppose $G$ contains a cycle $C$ which is either an end block or an isolated cycle. If $C$ contains a cut vertex call it $u$. Otherwise, designate any vertex of $C$ as $u$. If $C$ is a triangle choose some vertex $v$ arbitrarily from $G-C$. Remove $v$ as well as $C$ from $G$ and perform induction on the $n-4$
vertices that remain. As $n$ and $n-4$ have the same parity modulo $4$, this produces a $1$-sparse set of order $\Big\lfloor\dfrac{n-4}{2}\Big\rfloor+r(n)$. Add to this $1$-sparse set the two vertices of $C-u$ and obtain the desired result. Similarly, if $C$ is a $4$-cycle, apply induction to find a $1$-sparse set of order $\Big\lfloor\dfrac{n-4}{2}\Big\rfloor+r(n)$ in $G-C$ and add two vertices from $C-u$ to produce the desired result. So suppose $C$ has five or more
vertices. Let $P$ be a path on four vertices in $C$ that begin at $u$. As before, perform induction on $G-P$. Take the result and combine them with the two internal vertices of $P$ and produce the desired result.\\

\noindent So let us suppose that $G$ contains no end cycles. Let $P$ be a longest path in $G$. Suppose $P$ has length four or more. Let $u$ $(v)$ be the first (second) vertex of $P$ and $w$ $(x)$ be the penultimate (last). Remove these four vertices from $G$ and perform induction on what remains. This produces a $1$-sparse set and when we add to it $u$ and $x$, the desired result is obtained. Thus, every path in $G$ contains at most three vertices. Suppose $P$ is a path on exactly three vertices. Call them $u$, $v$ and $w$ where $v$ is the central vertex. Remove from $G$ these three vertices with one other vertex chosen arbitrarily from $G$. As above, perform induction then add to the $1$-sparse set $u$ and $w$ and to produce the desired result. If $G$ has no path of order three, then the entire graph is $1$-sparse and the desired result holds. \qed
\end{proof}

\begin{theorem}\label{thm:cactus1sparse}
Let $i\geq 5$ and $j\geq 1$. Then, $R_{1}^{\mathcal{CA}}(i,j)=\begin{cases}
2j-1, & \text{ if }j \text{ is odd} \\
2j-2, & \text{ if }j \text{ is even} 
\end{cases}$
\end{theorem}

\begin{proof}
Suppose $j$ is odd and $G$ is a cactus of order $2j-1$. As $2j-1$ isn't divisible by four, from Lemma \ref{lem:cactus1sparse}, we know it has a $1$-sparse set on at least $j$ vertices. Now, if we take the disjoint union of $\dfrac{j-1}{2}$ copies of $4$-cycles, clearly we obtain a cactus on $2j-2$ vertices which has no $1$-sparse $j$-set. Also, this graph has no $1$-dense $i$-set from Lemma \ref{lem:cactuskdense} and we are done since $i\geq 5$.\\

\noindent Suppose $j$ is even and $G$ is a cactus of order $2j-2$. Note, $2j-2$ isn't divisible by four. Again, proceed with Lemma \ref{lem:cactus1sparse} and note $G$ has a $1$-sparse set of order $j$. Now, if we take the disjoint union of $\dfrac{j-2}{2}$ copies of $4$-cycles and an isolated vertex, clearly we obtain a cactus graph on $2j-3$ vertices which has no $1$-sparse $j$-set. Also, this graph has no $1$-dense $i$-set from Lemma \ref{lem:cactuskdense} and we are done since $i\geq 5$. \qed \\
\end{proof}

\noindent Now, we will examine $k$-defective case for $k\geq 2$.

\begin{lemma}\label{lem:foresttocactus}
Every cactus contains a set of independent edges whose removal produces a forest.
\end{lemma}

\begin{proof}
We will proceed by induction on the number of cycles. If $G$ is a cactus with no cycles the statement is obviously true. So suppose $G$ has at least one cycle and suppose the statement is true for all cacti with fewer cycles. Let $A$ be a component of $G$ containing a cycle. If $A$ has exactly one cycle, remove one edge from it and proceed by induction on what remains and the desired result follows. So suppose $A$ has more than one cycle. Choose two that are furthest apart and call them $C$ and $C'$. Let $v$ be the vertex in $C$ that is closest to $C'$. Choose some edge, say $e$, of $C$ that isn't incident with $v$. Let $B$ be the component of $G-v$ that contains $e$. Perform induction on $G-B$ and find an independent set of edges whose removal from $G-B$ produces a forest. Insert $e$ into this set and the proof is complete. \qed
\end{proof}

\begin{corollary}\label{cor:cactusksparse}
Let $k\geq 2$. If $G$ is a cactus of order $n$, then $\alpha_k(G)\geq \Big\lceil\dfrac{k}{k+1}n\Big\rceil$.
\end{corollary}

\begin{proof}
Given $G$, a cactus of order $n$, let us invoke Lemma \ref{lem:foresttocactus} and find an
independent set of edges $E'$ whose removal leaves a spanning forest, say $F$. From Lemma \ref{lem:forestksparse}, we note $\alpha_{k-1}(F)\geq \Big\lceil\dfrac{k}{k+1}n\Big\rceil$. So, let $S$ be a $(k-1)$-sparse set of $ \Big\lceil\dfrac{k}{k+1}n\Big\rceil$ vertices of $F$. Note that each vertex of $S$ has degree at most $k$ in $G$ (since the edges in $E'$ contribute at most one to the degree of every vertex in $G$) and is thus $k$-sparse. \qed 
\end{proof}

\noindent We will slightly improve Corollary \ref{cor:cactusksparse} to obtain sharpness and consequently determine $R_k^{\mathcal{CA}}(i,j)$ for all $j,k\geq 2$ and $i\geq k+4$.

\begin{lemma}\label{lem:cactusksparse}
Let $k\geq 2$. If $G$ is a cactus of order $n$, then $\alpha_k(G)\geq \Big\lfloor\dfrac{k}{k+1}n\Big\rfloor+1$.
\end{lemma}

\begin{proof}
Fix $k\geq2$. We will prove by strong induction on $n$. Firstly, if $n\leq k+1$, then there is nothing to prove. Also, if $k+2\leq n\leq 2k+1$, then the result follows from Corollary \ref{cor:cactusksparse}.\\

\noindent Assume $n=2k+2$. We need to show that $G$ has a vertex $u$ such that $G-u$ is $k$-sparse. Let $\Delta(G)$ be the degree of a maximum degree vertex in $G$. If $\Delta(G)\leq k$, then it is trivial. So assume $\Delta(G) \geq  k+1$ and let $u$ be a vertex of maximum degree. If $d(u)\geq k+2$, then we get $|G-N[u]|\leq k-1$. Since $G$ is a cactus, $N(u)$ is $1$-sparse, thus  any vertex in $N(u)$ has at most $1+(k-1)=k$ neighbors in $G-u$. Moreover, since $G$ is a cactus, any vertex in $G-N[u]$ is adjacent to at most two vertices in $N(u)$, so its degree is at most $2+(k-2)=k$ in $G-u$. Thus, $G-u$ is $k$-sparse.\\ 

\noindent So, assume $d(u)=k+1$. Then we get $|G-N[u]|=k$. Take a vertex $w\in G-N[u]$. Note that if $d(w)=k+1$, then since $G$ is a cactus, it has exactly two neighbors from $N(u)$ and it is adjacent to all vertices in $G-N[u]$. Let $v_1,v_2\in N(u)$ be the neighbors of $w$. We note that $N(u)-\{v_1\}$ and $(G-N[u])-\{w\}$ are $1$-sparse sets (since $G$ is a cactus). Moreover, since $\{u,v_1,w,v_2\}$ forms a cycle and $w$ is adjacent to all vertices in $G-N[u]$, there is only one edge between $N(u)-\{v_1\}$ and $G-N[u]$, namely $v_2w$. Thus, $G-v_1$ forms a $k$-sparse set and we are done. So, assume $d(w)\leq k$ for all $w\in G-N[u]$. We note that any vertex in $N(u)$ has at most $k$ neighbors in $G-u$ since $\Delta(G)=k+1$. Then, $G-u$ is $k$-sparse and we are done.\\

\noindent If $2k+3\leq n\leq 3k+2$, then again the result follows from Corollary \ref{cor:cactusksparse}. Hence let us assume the lemma holds for all cacti on less than $n\geq 3k+3$ vertices. Let $G$ be a cactus graph on exactly $n\geq 3k+3$ vertices. Firstly, we observe that if a vertex of $G$ has degree at least three, then it is a cut vertex. Since $k\geq 2$, a vertex $u$ with $d(u)\geq k+1$ is necessarily a cut vertex. Thus, if $G$ has at most two cut vertices, all vertices but the two cut vertices form a $k$-sparse $(n-2)$-set and the result follows. Besides, if $G$ is disconnected, take a component $D$ on $m$ vertices and apply induction on $G-D$ and $D$.  We obtain a $k$-sparse set of $G$ of order at least $\Bigg(\Big\lfloor\dfrac{(n-m)k}{k+1}\Big\rfloor+1\Bigg)+\Bigg(\Big\lfloor\dfrac{mk}{k+1}\Big\rfloor+1\Bigg)\geq \Big\lfloor\dfrac{nk}{k+1}\Big\rfloor+1$ and we are done. Therefore, we can assume $G$ is connected and has at least three cut vertices. Throughout the remainder of this proof, we will often use the fact that non-cut vertices have degree at most 2 (thus at most $k$) in $G$, implying that any subset of vertices with no cut vertex in $G$ is a $k$-sparse set. \\

\noindent Let $u$ and $v$ be two cut vertices of $G$ which are furthest apart. If $u$ and $v$ are adjacent, then all cut vertices form a clique. Since $G$ has no clique of size 4, it follows that $G$ has exactly three cut vertices, say $w$ for the third. Now, $G-\{u,v,w\}$ is $k$-sparse and we are done if $n\geq 3k+4$. Hence assume $n=3k+3$. Let $A=N(u)-\{v,w\}$, $B=N(v)-\{u,w\}$, $C=N(w)-\{u,v\}$. Note that $A$, $B$ and $C$ are pairwise disjoint and $|A\cup B\cup C|\leq 3k$. Hence, at least one of $A$, $B$ and $C$ has at most $k$ vertices, without loss of generality say $A$. Thus, $G-\{v,w\}$ is $k$-sparse and once again we are done.\\

\noindent In what follows, we can assume the two further apart cut vertices $u$ and $v$ are not adjacent. Note that since $u$ and $v$ are farthest apart, exactly one connected component of $G-u$ have cut vertices. Let $L_u$ be the union of connected components of $G-u$ which have no cut vertex, and $|L_u|=l$. \\

\noindent If $l\geq k$, perform induction on $(G-u)\backslash L_u$, and take a $k$-sparse set of size at least $\Big\lfloor\dfrac{(n-l-1)k}{k+1}\Big\rfloor+1$, say $J$. Observe that $J\cup L_u$ is $k$-sparse and we get $$|J\cup L_u|\geq \Big\lfloor\dfrac{(n-l-1)k}{k+1}\Big\rfloor+1+l=\Big\lfloor\dfrac{nk+l-k}{k+1}\Big\rfloor+1\geq\Big\lfloor\dfrac{nk}{k+1}\Big\rfloor+1.$$
\noindent If $d(u)\leq k$, perform induction on $G\backslash L_u$, and take a $k$-sparse set of size at least $\Big\lfloor\dfrac{(n-l)k}{k+1}\Big\rfloor+1$, say $J$. Since $u$ has at most $k$ neighbors, $J\cup L_u$ is $k$-sparse (even if $u\in J$). Hence we get $$|J\cup L_u|\geq \Big\lfloor\dfrac{(n-l)k}{k+1}\Big\rfloor+1+l=\Big\lfloor\dfrac{nk+l}{k+1}\Big\rfloor+1\geq\Big\lfloor\dfrac{nk}{k+1}\Big\rfloor+1.$$

\noindent As we are done in both of the above cases, we can assume $d(u)\geq k+1$ and $l\leq k-1$. Since $G-(L_u \cup \{u\})$ is connected and $G$ is a cactus, $u$ has at most two neighbors in $G-L_u$, thus $d(u)\leq (k-1)+2$. It follows that $l=k-1$ and $d(u)=k+1$ where $u$ is adjacent to all vertices in $L_u$ and has exactly two neighbors in $G-L_u$. Let $x$ and $y$ be the neighbors of $u$ in $G-L_u$. Since $x$ and $y$ belong to the same connected component of $G-(L_u \cup \{u\})$ and $u$ is adjacent to both $x$ and $y$, the fact that $G$ is a cactus implies that no vertex in $L_u$ has a neighbor in  $G-(L_u \cup \{u\})$. \\

\noindent If $x$ has at most $k-1$ neighbors in $G-\{u,y\}$ then we are done from what follows. Apply induction on $G-(L_u\cup\{u,y\})$ to obtain a $k$-sparse set of size at least $\Big\lfloor\dfrac{(n-k-1)k}{k+1}\Big\rfloor+1$, say $J$. Observe that $J\cup L_u\cup \{u\}$ is $k$-sparse since the only edge between $J$ and $L_u\cup\{u\}$ is $ux$, and $x$ has at most $k$ neighbors in $J\cup L_u \cup\{u\}$. Hence we get $$|J\cup L_u \cup\{u\}|\geq \Big\lfloor\dfrac{(n-k-1)k}{k+1}\Big\rfloor+1+k=\Big\lfloor\dfrac{nk}{k+1}\Big\rfloor+1.$$       

\noindent So, we can assume $x$ has at least $k$ neighbors in $G-\{u,y\}$, thus $d(x)\geq k+1$ and $x$ is a cut vertex. By symmetry, we also assume $y$ has at least $k$ neighbors in $G-\{u,x\}$, thus $d(y)\geq k+1$ and $y$ is a cut vertex.

The following summarizes the assumptions which are valid till the end of the proof. Let $d$ be the distance between the furthest apart two cut vertices $u$ and $v$. Then we have:
\begin{assumption}\label{assum:cutvertex}
Every cut vertex $z$ of distance $d$ from $v$ has the following properties:
\begin{enumerate}
\item[i)] $d(z)=k+1$,
\item[ii)] Let $L_z$ be the union of the connected components with no cut vertex in $G-z$. Then $|L_z|=k-1$ and no vertex in $L_z$ has a neighbor in $G-(L_z\cup \{z\})$,
\item[iii)] $z$ is adjacent to all vertices in $L_z$ and has exactly two neighbors in $G-L_z$, say $x_z$ and $y_z$, 
\item[iv)] $x_z$ (resp. $y_z$) is a cut vertex with at least $k$ neighbors in $G-\{z,y_z\}$ (resp. in $G-\{z,x_z\}$).
\end{enumerate}
\end{assumption}

Now, if $x$ and $y$ are adjacent, take a shortest path between $u$ and $v$.  Without loss of generality, this path passes through $x$. Hence any path between $v$ and $y$ has to pass through $x$ because otherwise we would get two cycles that intersect on the edge $xy$. Note that the distance between $x$ and $v$ is $d-1$, hence $y$ is a cut vertex of distance $d$ from $v$. Then we obtain a contradiction with Assumption \ref{assum:cutvertex} i) as $d(y)\geq k+2$. 

So, we assume in what follows that $x$ and $y$ are not adjacent. By definition of distance, both $x$ and $y$ are of distance at least $d-1$ to $v$. Moreover, both $x$ and $y$ are of distance at most $d$ from $v$ since $d$ is the maximum distance between two cut vertices of $G$. Now, if one of $x$ or $y$, say without loss of generality $y$, is of distance $d$ from $v$, then $y$ is a cut vertex satisfying Assumption \ref{assum:cutvertex}. Let $w$ be the neighbor of $y$ in $G-L_y$ other than $u$. Apply induction on $G-(L_u\cup L_y \cup \{u,y,x,w\})$ to obtain a $k$-sparse set of size at least $\Big\lfloor\dfrac{(n-2k-2)k}{k+1}\Big\rfloor+1$, say $J$. Observe that $J\cup L_u\cup L_y \cup \{u,y\}$ is $k$-sparse since there is no edge between $J$ and $L_u\cup L_y \cup \{u,y\}$ by Assumption \ref{assum:cutvertex} ii),  and each one of $u$ and $y$ has degree $k$ in  $J\cup L_u\cup L_y \cup \{u,y\}$. Hence we get $$|J\cup L_u\cup L_y \cup \{u,y\}|\geq \Big\lfloor\dfrac{(n-2k-2)k}{k+1}\Big\rfloor+1+2k=\Big\lfloor\dfrac{nk}{k+1}\Big\rfloor+1.$$       

So assume $v$ has distance $d-1$ with both of $x$ and $y$. Take a shortest path between $x$ (resp. $y)$ and $v$, say $x_1$ (resp. $y_1)$ is the neighbor of $x$ (resp. $(y)$ in this path. Clearly this path does not pass through $u$ since $u$ is on distance $d$ from $v$. Note that shortest paths from $v$ to $x$ and $v$ to $y$ can intersect and we can possibly have $x_1=y_1$, $x_1=v$ or $y_1=v$.\\

\noindent Since $xy \notin E$ and by Assumption \ref{assum:cutvertex} iv), we have $|N(x)-\{u,x_1\}|\geq k-1$. Consider a vertex $w\in N(x)-\{u,x_1\}$ and observe that any path between $v$ and $w$ contains $x$, thus the distance from $w$ to $v$ is $d$. If $w$ is a cut vertex, then by Assumption \ref{assum:cutvertex} ii) the set $L_w$ (which has no cut vertices) has size $k-1$. Assumption \ref{assum:cutvertex} iii) implies that $w$ has exactly one neighbor 
in $G - (L_w\cup \{x\})$, say $t$. If $t$ is a cut vertex then $t$ is adjacent to $x$ since its distance to $v$ is at most $d$, and by Assumption \ref{assum:cutvertex} i), we have $d(t)=k+1$. If $t$ is not a cut vertex, then we know $d(t)\leq 2$. In both cases, $t$ has at most $k$ neighbors other than $x$. Hence, apply induction on $G-(L_w\cup N[u]\cup\{w\})$, and take a $k$-sparse set of size at least $\Big\lfloor\dfrac{(n-2k-2)k}{k+1}\Big\rfloor+1$, say $J$. Observe that $t$ has at most $k$ neighbors in  $G-x$, and by Assumption \ref{assum:cutvertex} ii), the only edge between $J$ and $L_w\cup L_u\cup \{u,w\}$ is $wt$. Therefore, $J\cup L_w\cup L_u\cup \{u,w\}$ is $k$-sparse. Hence we get $$|J\cup L_w\cup L_u\cup \{u,w\}|\geq \Big\lfloor\dfrac{(n-2k-2)k}{k+1}\Big\rfloor+1+(k-1)+(k-1)+2=\Big\lfloor\dfrac{nk}{k+1}\Big\rfloor+1.$$

\noindent It remains to consider the case where none of the vertices in $N(x)-\{u,x_1\}$ and $N(y)-\{u,y_1\}$ are cut vertices.  Let $A_x=N(x)-\{u,x_1\}$ and $A_y=N(y)-\{u,y_1\}$ with $|A_x|=a_x$ and $|A_y|=a_y$. Note that any vertex $x_2\in A_x$ (respectively, $y_2\in A_y$) has degree at most two. Moreover, if there is a neighbor of $x_2$ (respectively, $y_2)$, say $x_3$ (respectively, $y_3)$, which is not in $N[x]$ (respectively, $N[y])$, then we get $d(x_3)\leq k$ (respectively, $d(y_3)\leq k)$ because otherwise $x_3$ (respectively, $y_3)$ would be a cut vertex with distance $d+1$ to $v$. Now, apply induction on $G-(A_x\cup A_y \cup L_u\cup\{u,x,y\})$, take a $k$-sparse set of size at least $\Big\lfloor\dfrac{(n-k-2-a_x-a_y)k}{k+1}\Big\rfloor+1$, say $J$. Note that any edge between $J$ and $A_x\cup A_y \cup L_u\cup\{u\}$ is incident to a vertex in $A_x\cup A_y$. Since any neighbor of a vertex in $A_x\cup A_y$ has degree at most $k$, the set $J\cup A_x\cup A_y \cup L_u\cup\{u\}$ is $k$-sparse. Hence we get
\begin{eqnarray*}
|J\cup A_x\cup A_y \cup L_u\cup\{u\}|&\geq& \Big\lfloor\dfrac{(n-k-2-a_x-a_y)k}{k+1}\Big\rfloor+1+a_x+a_y+(k-1)+1\\
&=& \Big\lfloor\dfrac{nk+a_x+a_y-k}{k+1}\Big\rfloor+1 \\
&\geq& \Big\lfloor\dfrac{nk}{k+1}\Big\rfloor+1
\end{eqnarray*}
\noindent since $a_x,a_y\geq k-1$ and so we have $a_x+a_y\geq 2k-2\geq k$. As a result, in all cases, we have $\alpha_k(G)\geq \Big\lfloor\dfrac{nk}{k+1}\Big\rfloor+1$ and we are done. \qed
\\
\end{proof}

\noindent The following result will be useful when constructing an extremal graph in Theorem \ref{thm:cactus}.

\begin{lemma}\label{lem:extremalcactus}
	Let $k\geq 2$. Define $G_{k,0}$ as $K_{1,k+1}$, and $G_{k,l}$ as in Figure \ref{fig:cactusextremal} for $l\geq1$. Then, $G_{k,l}$ is a cactus graph on $(k+1)(l+1)+1$ vertices with $\alpha_k(G_{k,l})=k(l+1)+1$ for all $l\geq 0$.
\end{lemma}

\begin{proof}
We proceed by induction on $l$. If $l=0$, it is trivial. Assume $\alpha_k(G_{k,l-1})=kl+1$ for some $l\geq 1$ and consider $G_{k,l}$. Take the unique vertex $v\in G_{k,l}$ of degree $k+1$. Observe that $v$ has exactly $k-1$ neighbors of degree one, and exactly two neighbors of degree two. Let $x$ and $y$ be the neighbors of $v$ with degree two. Note that $G_{k,l}-\{N[v]-\{x\}\}$ is isomorphic to $G_{k,l-1}$, hence it has a $k$-sparse set of size $kl+1$, say $J_1$. So we obtain $\alpha_k(G_{k,l})\geq kl+1+k=k(l+1)+1$ since $J_1\cup (N[v]-\{x,y\})$ is $k$-sparse.\\

\noindent Let us now prove that $\alpha_k(G_{k,l})\leq kl+1+k=k(l+1)+1$. Assume for a contradiction $\alpha_k(G_{k,l})\geq k(l+1)+2$ and consider a $k$-sparse set on $k(l+1)+2$ vertices, say $J_2$. By the induction hypothesis, we have $|J_2\cap (G_{k,l}-(N[v]-\{x\}))|\leq kl+1$. This implies $|J_2|=k(l+1)+2$, we get $|J_2\cap (N[v]-\{x\})|\geq k+1$. Since there are exactly $k+1$ vertices in $N[v]-\{x\}$, we have $(N[v]-\{x\}) \subseteq J_2$. However, $v$ can have at most $k$ neighbors in $J_2$ and $v$ has exactly $k$ neighbors in $N[v]-\{x\}$. It follows that $x\notin J_2$. Note that $G_{k,l}-\{N[v]-\{y\}\}$ is also isomorphic to $G_{k,l-1}$. By symmetry, this implies $y\notin J_2$. However, this contradicts  $(N[v]-\{x\}) \subseteq J_2$ and the result follows. \qed

\end{proof}

\def\r{4pt}
\def\dy{1cm}
\tikzset{c/.style={draw,circle,fill=white,minimum size=\r,inner sep=0pt,
		anchor=center},
	d/.style={draw,circle,fill=black,minimum size=\r,inner sep=0pt, anchor=center}}

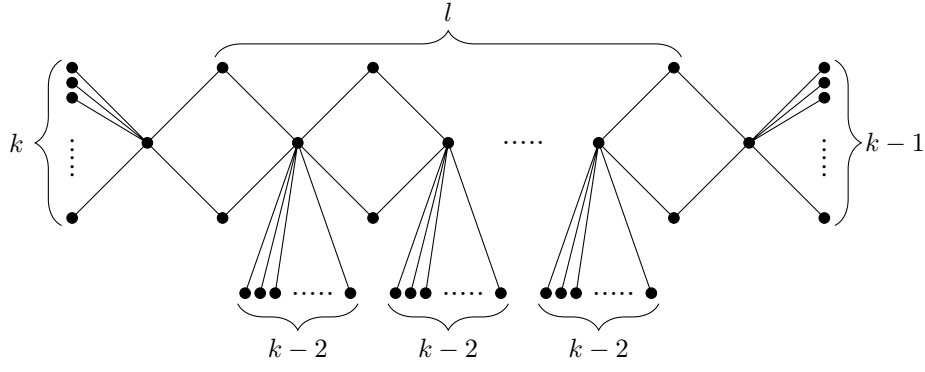
\begin{figure}[htbp]
	\begin{center}
		
		\begin{tikzpicture}

		\node[d] at (0,2) {};
		\node[d] at (1,1) {};
		\node[d] at (0,0) {};
		\node[d] at (-1,1) {};
		\draw (0,2) to (1,1)
		(0,2) to (-1,1)
		(0,0) to (1,1)
		(0,0) to (-1,1);
		
		\node[d] at (0.3,-1) {};
		\node[d] at (0.5,-1) {};
		\node[d] at (0.7,-1) {};
		\node[d] at (1.7,-1) {};
		\draw (1,1) to (0.3, -1)
		(1,1) to (0.5, -1)
		(1,1) to (0.7, -1)
		(1,1) to (1.7, -1);
		\node[thick,minimum size=3cm] at (1.2,-1) {.....};
		
		\node[d] at (2,2) {};
		\node[d] at (3,1) {};
		\node[d] at (2,0) {};
		\draw (1,1) to (2,2)
		(1,1) to (2,0)
		(3,1) to (2,2)
		(3,1) to (2,0);
		
		\node[d] at (2.3,-1) {};
		\node[d] at (2.5,-1) {};
		\node[d] at (2.7,-1) {};
		\node[d] at (3.7,-1) {};
		\draw (3,1) to (2.3, -1)
		(3,1) to (2.5, -1)
		(3,1) to (2.7, -1)
		(3,1) to (3.7, -1);
		\node[thick,minimum size=3cm] at (3.2,-1) {.....};

		\node[thick,minimum size=3cm] at (4,1) {.....};
		
		\node[d] at (6,2) {};
		\node[d] at (5,1) {};
		\node[d] at (6,0) {};
		\node[d] at (7,1) {};
		\draw (5,1) to (6,2)
		(5,1) to (6,0)
		(7,1) to (6,2)
		(7,1) to (6,0);
		
		\node[d] at (4.3,-1) {};
		\node[d] at (4.5,-1) {};
		\node[d] at (4.7,-1) {};
		\node[d] at (5.7,-1) {};
		\draw (5,1) to (4.3, -1)
		(5,1) to (4.5, -1)
		(5,1) to (4.7, -1)
		(5,1) to (5.7, -1);
		\node[thick,minimum size=3cm] at (5.2,-1) {.....};

		\node[d] at (8,2) {};
		\node[d] at (8,1.8) {};
		\node[d] at (8,1.6) {};
		\node[d] at (8,0) {};
		\draw (7,1) to (8,2)
		(7,1) to (8,1.8)
		(7,1) to (8,1.6)
		(7,1) to (8,0);
		\node[thick,minimum size=3cm, rotate=90] at (8,0.8) {.....};

		\node[d] at (-2,2) {};
		\node[d] at (-2,1.8) {};
		\node[d] at (-2,1.6) {};
		\node[d] at (-2,0) {};
		\draw (-1,1) to (-2,2)
		(-1,1) to (-2,1.8)
		(-1,1) to (-2,1.6)
		(-1,1) to (-2,0);
		\node[thick,minimum size=3cm, rotate=90] at (-2,0.8) {.....};
		\draw [decorate,decoration={brace,amplitude=10pt},xshift=-4pt,yshift=0pt]
		(-2,-0.1) -- (-2,2.1) node [black,midway,xshift=-0.6cm] 
		{\footnotesize $k$};
		
		\draw [decorate,decoration={brace,amplitude=10pt,mirror},xshift=0pt,yshift=-4pt]
		(0.2,-1) -- (1.8,-1) node [black,midway,yshift=-0.6cm] 
		{\footnotesize $k-2$};
		
		\draw [decorate,decoration={brace,amplitude=10pt,mirror},xshift=0pt,yshift=-4pt]
		(2.2,-1) -- (3.8,-1) node [black,midway,yshift=-0.6cm] 
		{\footnotesize $k-2$};
		
		\draw [decorate,decoration={brace,amplitude=10pt,mirror},xshift=0pt,yshift=-4pt]
		(4.2,-1) -- (5.8,-1) node [black,midway,yshift=-0.6cm] 
		{\footnotesize $k-2$};
		
		\draw [decorate,decoration={brace,amplitude=10pt,mirror},xshift=4pt,yshift=0pt]
		(8,-0.1) -- (8,2.1) node [black,midway,xshift=0.8cm] 
		{\footnotesize $k-1$};
		
		\draw [decorate,decoration={brace,amplitude=10pt},xshift=0pt,yshift=4pt]
		(-0.1,2) -- (6.1,2) node [black,midway,yshift=0.6cm] 
		{\footnotesize $l$ };

		\end{tikzpicture}
		\vspace{-1cm}
	\end{center}
	\caption{The graph $G_{k,l}$ for $k\geq 2$ and $l\geq 1$.} \label{fig:cactusextremal}
\end{figure}

\begin{theorem}\label{thm:cactus}
	Let $j,k\geq 2$ and $i\geq k+4$. Then, $R_k^{\mathcal{CA}}(i,j)=j-1+\Big\lceil\dfrac{j-1}{k}\Big\rceil$.
\end{theorem}

\begin{proof}
	Firstly, if $j\leq k+1$, then we get $j-1+\Big\lceil\dfrac{j-1}{k}\Big\rceil=j$ and the result follows from Remark \ref{rem:generalsmallvaluesofiandj}. For $j\geq k+2$, let $j-1=ks+t$ for some $s\geq 1$ and $0\leq t\leq k-1$.

If $t\neq 0$, then we get $j-1+\Big\lceil\dfrac{j-1}{k}\Big\rceil=ks+t+s+1$. Let  $G$ be a cactus on $ks+t+s+1$ vertices. By Corollary \ref{cor:cactusksparse}, we get $$\alpha_k(G)\geq \Big\lceil\dfrac{k(ks+t+s+1)}{k+1}\Big\rceil=ks+\Big\lceil\dfrac{k(t+1)}{k+1}\Big\rceil= ks+\Big\lceil\dfrac{(k+1)t+(k-t)}{k+1}\Big\rceil =  ks+t+1=j.$$
To construct an extremal graph $H$, take the disjoint union of $G_{k,s-1}$ and $t-1$ isolated vertices. Observe that $H$ has $((k+1)s+1)+(t-1)=ks+t+s$ vertices. Thus, Lemma \ref{lem:extremalcactus} implies  $\alpha_k(H)=(ks+1)+(t-1)=ks+t=j-1$. Moreover, $H$ has no $k$-dense $i$-set from Lemma \ref{lem:cactuskdense} since $i\geq k+4$.

If $t=0$, then we get $s\geq 2$ and $j-1+\Big\lceil\dfrac{j-1}{k}\Big\rceil=(k+1)s$. Let  $G$ be a cactus on $(k+1)s$ vertices. By Lemma \ref{lem:cactusksparse}, we get $$\alpha_k(G)\geq\Big\lfloor\dfrac{k(k+1)s}{k+1}\Big\rfloor+1=ks+1=j.$$
Now, let $H$ be the graph obtained by taking the disjoint union of $G_{k,s-2}$ and $k-1$ isolated vertices. Observe that $H$ has $((k+1)(s-1)+1)+(k-1)=(k+1)s-1$ vertices. Thus, it follows from Lemma \ref{lem:extremalcactus} that $\alpha_k(H)=(k(s-1)+1)+(k-1)=ks=j-1$. Moreover, $H$ has no $k$-dense $i$-set from Lemma \ref{lem:cactuskdense} since $i\geq k+4$. \qed \\
\end{proof}

\noindent To summarize, we have the formula for defective Ramsey numbers in cacti for all $i\geq k+4$ from Theorems \ref{thm:cactus1sparse} and \ref{thm:cactus} whenever $k\geq 1$, which leaves $R_k^{\mathcal{CA}}(k+3,j)$ for $j\geq k+3$ as the only open case with Remark \ref{rem:cactuskplustwo}. Observe that the graph $H_j$ in Figure \ref{fig:extremalfor1-4} is a cactus on $2j-3$ vertices which has no $1$-dense $4$-set and no $1$-sparse $j$-set. Thus, we get $R_1^{\mathcal{CA}}(4,j)\geq 2j-2$. On the other hand, if a cactus on four vertices is not $1$-dense (or equivalently, it is not a 4-cycle), clearly it has a $1$-sparse $3$-set. Therefore, Lemma \ref{lem:cactus1sparse} can be modified as $\alpha_1(G)\geq \Big\lfloor\dfrac{n}{2}\Big\rfloor+1$ with the same induction step whenever $G$ has no $1$-dense $4$-set. Hence, if $G$ is a cactus on $2j-2$ vertices with no $1$-dense $4$-set, we get $\alpha_1(G)\geq j$ and so $R_1^{\mathcal{CA}}(4,j)=2j-2$. As for $k=2$ and $k=3$, the extremal graph $G_{k,l}$ for $R_k^{\mathcal{CA}}(k+4,j)$ in Figure \ref{fig:cactusextremal} has no $k$-dense $(k+3)$-set. Together with $R_k^{\mathcal{CA}}(k+3,j)\leq R_k^{\mathcal{CA}}(k+4,j)$ from Remark \ref{rem:defectiveinequalityforiandj}, this implies that $R_k^{\mathcal{CA}}(k+3,j)\leq R_k^{\mathcal{CA}}(k+4,j)$. It follows that the formula in Theorem \ref{thm:cactus} is still valid whenever $i=k+3$ and $k\in\{1,2,3\}$.\\

\def\r{4pt}
\def\dy{1cm}
\tikzset{c/.style={draw,circle,fill=white,minimum size=\r,inner sep=0pt,
		anchor=center},
	d/.style={draw,circle,fill=black,minimum size=\r,inner sep=0pt, anchor=center}}

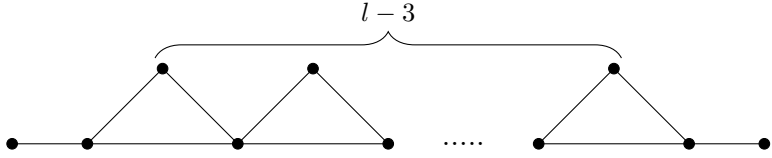
\begin{figure}[h]
	\begin{center}
		
		\begin{tikzpicture}

		\node[d] at (0,2) {};
		\node[d] at (1,1) {};
		\node[d] at (-1,1) {};
		\draw (0,2) to (1,1)
		(0,2) to (-1,1)
		(-1,1) to (1,1);
		
		\node[d] at (2,2) {};
		\node[d] at (3,1) {};
		\draw (1,1) to (2,2)
		(1,1) to (3,1)
		(3,1) to (2,2);
		
		\node[thick,minimum size=3cm] at (4,1) {.....};
		
		\node[d] at (6,2) {};
		\node[d] at (5,1) {};
		\node[d] at (7,1) {};
		\draw (5,1) to (6,2)
		(7,1) to (6,2)
		(7,1) to (5,1);
		
		\node[d] at (8,1) {};
		\draw (7,1) to (8,1);
		
		\node[d] at (-2,1) {};
		\draw (-1,1) to (-2,1);
		
		\draw [decorate,decoration={brace,amplitude=10pt},xshift=0pt,yshift=4pt]
		(-0.1,2) -- (6.1,2) node [black,midway,yshift=0.6cm] 
		{\footnotesize $l-3$};

		\end{tikzpicture}
		\vspace{-1.5cm}
	\end{center}
	\caption{The graph $H_{l}$ for $l\geq 4$.}\label{fig:extremalfor1-4}
\end{figure}

\noindent For $k\geq 4$, we will give a two sided bound for $R_k^{\mathcal{CA}}(k+3,j)$ to handle the case $j\leq 2k+1$. 
 
\begin{corollary}\label{cor:cactifinalize}
Let $k\geq 4$ and $j\geq k+2$. Then, $$j+\Big\lfloor\dfrac{j-1}{k+1}\Big\rfloor \leq R_k^{\mathcal{CA}}(k+3,j)\leq j-1+\Big\lceil\dfrac{j-1}{k}\Big\rceil$$
\end{corollary}

\begin{proof}
Since every forest is a cacti, the first inequality is a direct consequence of Remark \ref{rem:graphclasshereditary} and $R_k^{\mathcal{FO}}(k+3,j)=j+\Big\lfloor\dfrac{j-1}{k+1}\Big\rfloor$ (Theorem \ref{thm:forest}) whereas the second inequality comes from Remark \ref{rem:defectiveinequalityforiandj} and $R_k^{\mathcal{CA}}(k+4,j)=j-1+\Big\lceil\dfrac{j-1}{k}\Big\rceil$ (Theorem \ref{thm:cactus}). \qed \\
\end{proof}

\noindent By Corollary \ref{cor:cactifinalize}, we have $R_k^{\mathcal{CA}}(k+3,j)=j+1$ whenever $k+2\leq j\leq 2k+1$. As such, we settle all defective Ramsey numbers in cacti except $R_k^{\mathcal{CA}}(k+3,j)$ for $k\geq 4$ and $j\geq 2k+2$, and we leave this as an open question.\\



\section{Bipartite Graphs}
In this section, we present $1$-defective Ramsey numbers on bipartite graphs for all $i,j \geq 3$ except five specific values for which we provide a conjecture. In addition, we provide a formula for all $k\geq 2$ and $i\geq 2k+3$ values. Firstly, let us recall the classical Ramsey numbers. Let $\mathcal{BIP}$ be the class of bipartite graphs.

\begin{theorem}\label{thm:bipartitegeneralRamsey}
With the preceding notation, \cite{ramseygraphclassesHeggernes} $R_{0}^{\mathcal{BIP}}(i,j)=2j-1$ for all $i\geq 3$ and $j\geq1$.
\end{theorem}

\noindent Now, we show that $R_{1}^{\mathcal{BIP}}(i,j)$ does not depend on $i$ if $i\neq 4$ in the following theorem.

\begin{theorem}\label{thm:bipartiteexcept4}
The following hold:
\begin{itemize}
\item[(i)] $R_{1}^{\mathcal{BIP}}(3,j)=j$ for all $j\geq 3$.
\item[(ii)] $R_{1}^{\mathcal{BIP}}(i,j)=2j-1$ for all $i\geq 5$ and $j\geq 3$.
\end{itemize}
\end{theorem}

\begin{proof}
Since an empty graph is bipartite, $R_{1}^{\mathcal{BIP}}(3,j)=j$ directly follows from Lemma \ref{lem:generalkplustwo}. On the other hand, Theorem \ref{thm:bipartitegeneralRamsey} and Remark \ref{rem:defectiveinequalityfork} give $R_{1}^{\mathcal{BIP}}(i,j)\leq R_{0}^{\mathcal{BIP}}(i,j)=2j-1$ for all $i\geq 5$ and $j\geq 3$. Note that, for $i\geq 5$, any $1$-dense $i$-set has a triangle. Thus, a bipartite graph does not contain a $1$-dense $i$-set. Since $K_{j-1,j-1}$ has no $1$-sparse $j$-set for $j\geq 3$, we get $R_{1}^{\mathcal{BIP}}(i,j)\geq 2j-1$ for all $i\geq 5$ and $j\geq 3$. Hence, the result follows. \qed \\
\end{proof}

\noindent When $i=4$, we first establish the following singular values in order to derive a general formula for $R_{1}^{\mathcal{BIP}}(4,j)$.

\begin{theorem}\label{thm:bipartitesmall}
	Each of the following hold:
	\begin{itemize}
 
		\item[(i)] If $j\in\{4,5,6\}$, then $R_{1}^{\mathcal{BIP}}(4,j)=2j-3$.
		\item[(ii)] If $j\in\{3,7\}$, then $R_{1}^{\mathcal{BIP}}(4,j)=2j-2$.
		\item[(iii)] If $j\in\{8,9,13,14\}$, then $R_{1}^{\mathcal{BIP}}(4,j)=2j-1$.
	\end{itemize}
\end{theorem}

\begin{proof}
	\begin{itemize}
		\item[(i)] Firstly, we observe that none of $K_{1,3}$, $C_6$ and $C_8$ contains a $1$-dense $4$-set. Secondly, $K_{1,3}$ has no $1$-sparse $4$-set, and $C_6$ has no $1$-sparse $5$-set, and $C_8$ has no $1$-sparse $6$-set. Therefore, we have $R_{1}^{\mathcal{BIP}}(4,j)\geq 2j-3$ for $j\in\{4,5,6\}$. 

Now, let $G$ be a bipartite graph with bipartition $(A,B)$ on $2j-3$ vertices where $j\in\{4,5,6\}$. We will show that $G$ contains a 1-dense 4-set or a 1-sparse $j$-set. If $|A|$ or $|B|$ is at least $j$, we are done. Thus, we can assume without loss of generality $|A|=j-2$ and $|B|=j-1$ since $|A|+|B|=2j-3$. If a vertex in $A$ has degree one then this vertex together with $B$ forms a 1-sparse $j$-set, and we are done. So, assume that each vertex in $A$ has degree at least two. We consider two complementary cases:
		\begin{itemize}
			\item Assume there exists a vertex $a\in A$ of degree at least $3$, say $u,v,w\in B$ are adjacent to $a$. Consider the union of $A-a$ and $\{u,v,w\}$. Either $(A-a) \cup \{u, v, w\} $ is a $1$-sparse $j$-set and we are done or at least one of $u,v,w$ has at least two neighbors in $A-a$. Without loss of generality, assume $b$ and $c$ are two neighbors of $u$ in $A-a$. Now, each of $b$ and $c$ has at least one neighbor in $B$ other than $u$, say $b$ is adjacent to $x$ and $c$ is adjacent to $y$. We can assume that $N(b)\cap \{y,v,w\} = N(c)\cap \{x,v,w\} = \emptyset$ and $x,y,v$ and $w$ are all distinct, or else a 1-dense 4-set is formed and we are done. Thus, we get $|B|=j-1\geq 5$, implying $j=6$. In this case, $\{b,c,v,w,x,y\}$ forms a $1$-sparse $6$-set and we are done.
			\item Assume every vertex in $A$ has degree exactly two. If the neighbors of vertices in $A$ are pairwise disjoint, then we would get $2(j-2)\leq j-1$ and so $j\leq 3$. Hence there exist two vertices $a,b\in A$ with a common neighbor $u$. Moreover, each one of $a$ and $b$ has exactly one other neighbor in $B-u$ and we can assume that these neighbors are distinct (or else they form a 1-dense 4-set with $a$ and $b$ and we are done).  Thus, $\{a,b\} \cup B -u$ is a $1$-sparse $j$-set, and we are done.  
		\end{itemize}   
		\item[(ii)] Observe that $K_{1,2}$ has no $1$-dense $4$-set and no $1$-sparse $3$-set. Moreover, a bipartite graph on $4$ vertices with no $1$-sparse $3$-set, contains a $C_4$, which is a $1$-dense $4$-set. Hence we get $R_{1}^{\mathcal{BIP}}(4,3)= 4$. On the other hand, we have $R_{1}^{\mathcal{BIP}}(4,7)=12$ from \cite{1defectiveperfectTinazOylumJohn} since bipartite graphs are exactly triangle-free perfect graphs.   
		\item[(iii)] Let $j\in\{8,9,13,14\}$. Firstly, we have $R_{1}^{\mathcal{BIP}}(4,j)\leq R_{0}^{\mathcal{BIP}}(4,j)= 2j-1$ from Remark \ref{rem:defectiveinequalityfork} and Theorem \ref{thm:bipartitegeneralRamsey}. We will complete the proof by constructing a bipartite graph on $2j-2$ vertices which has no $1$-dense $4$-set and no $1$-sparse $j$-set for each $j\in \{8,9,13,14\}$.
		\def\r{4pt}
		\def\dy{1cm}
		\tikzset{c/.style={draw,circle,fill=white,minimum size=\r,inner sep=0pt,
				anchor=center},
			d/.style={draw,circle,fill=black,minimum size=\r,inner sep=0pt, anchor=center}}
		
		\begin{figure}[h]
		\begin{center}
			\begin{tikzpicture}
			\pgfmathtruncatemacro{\Ncorners}{14}
			
			\node[draw, thick,rotate=90,minimum size=5cm,regular polygon, regular polygon sides=14] at (0,0) 
			(poly\Ncorners) {};
			
			\foreach\x in {1,...,\Ncorners}{
				\node[d] (poly\Ncorners-\x) at (poly\Ncorners.corner \x){};
				\node[anchor=\x*(360/\Ncorners)]at(poly\Ncorners.corner \x){$a_{\x}$};
			}
			\foreach\X in {1,...,\Ncorners}{
				\foreach\Y in {1,...,\Ncorners}{
					\pgfmathtruncatemacro{\Z}{abs(mod(abs(\Ncorners+\X-\Y),\Ncorners)-9)}
					\pgfmathtruncatemacro{\D}{abs(mod(abs(\X),2)-1)}
					\ifnum\Z=0
					\ifnum\D=0
					\draw (poly\Ncorners-\X) to (poly\Ncorners-\Y);
					\fi
					\fi
				}
			}
			\end{tikzpicture}
			\begin{tikzpicture}
			
			\pgfmathtruncatemacro{\Ncorners}{16}
			
			\node[draw, thick,rotate=90,minimum size=5cm,regular polygon, regular polygon sides=16] at (0,0) 
			(poly\Ncorners) {};
			
			\foreach\x in {1,...,\Ncorners}{
				\node[d] (poly\Ncorners-\x) at (poly\Ncorners.corner \x){};
				\node[anchor=\x*(360/\Ncorners)]at(poly\Ncorners.corner \x){$b_{\x}$};
			}
			\foreach\X in {1,...,\Ncorners}{
				\foreach\Y in {1,...,\Ncorners}{
					\pgfmathtruncatemacro{\Z}{abs(mod(abs(\Ncorners+\X-\Y),\Ncorners)-11)}
					\pgfmathtruncatemacro{\D}{abs(mod(abs(\X),2)-1)}
					\ifnum\Z=0
					\ifnum\D=0
					\draw (poly\Ncorners-\X) to (poly\Ncorners-\Y);
					\fi
					\fi
				}
			}
			\end{tikzpicture}
			\caption{Extremal graphs $G_1$ (on the left) and $G_2$ (on the right) for $R_{1}^{\mathcal{BIP}}(4,8)=15$  and $R_{1}^{\mathcal{BIP}}(4,9)=17$  respectively.} \label{fig:bipartite8-9}
	
\end{center}	
	\end{figure}
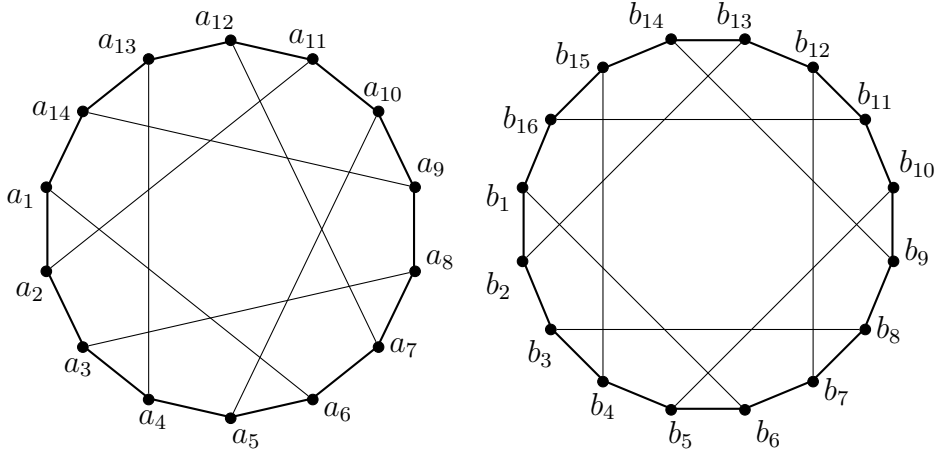

\noindent For $j=8$ (resp. $j=9$), consider the graph $G_1$ (resp. $G_2$) in Figure \ref{fig:bipartite8-9}. We note that $G_1$ is the well-known Heawood graph. Let $O$ and $E$ be the set of vertices whose indices are odd and even, respectively. Both $G_1$ and $G_2$ are 3-regular bipartite graphs on $14$ (resp. $16$) vertices with bipartition $(O,E)$ and no $1$-dense $4$-set. We claim that $G_1$ (resp. $G_2$) has no $1$-sparse $8$-set (resp. no $1$-sparse $9$-set). Assume the contrary and take a $1$-sparse $8$-set (resp. $1$-sparse $9$-set) $S=O_1\cup E_1$ where $O_1\subseteq O$ and $E_1\subseteq E$. Since both $G_1$ and $G_2$ are symmetric, assume without loss of generality that $|E_1|=x\leq 4$. Note that $|O|=7$ (resp $|O|=8$) and $|O_1|=8-x$ (resp. $|O_1|=9-x$) give $|O-O_1|=x-1$. Since $S$ is $1$-sparse and both $G_1$ and $G_2$ are 3-regular, each vertex in $E_1$ has at least two neighbors in $O-O_1$. Therefore, there are at least $2x$ edges between $E_1$ and $O-O_1$. It follows by the pigeonhole principle that there are two vertices $u,v\in O-O_1$ such that each of $u$ and $v$ has exactly three neighbors in $E_1$. Since $x\leq 4$, this implies that $u$ and $v$ have at least two common neighbors, contradicting the fact that $G_1$ (resp. $G_2$) has no $1$-dense $4$-set. As a result, we get $R_{1}^{\mathcal{BIP}}(4,8)=15$ (resp. $R_{1}^{\mathcal{BIP}}(4,9)=17$).\\

\begin{figure}[h]
	\begin{center}

		\begin{tikzpicture}
		
		\pgfmathtruncatemacro{\Ncorners}{26}
		
		\node[thick,rotate=90,minimum size=5cm,regular polygon, regular polygon sides=26] at (0,0)
		(poly\Ncorners) {};
		\foreach\x in {1,...,24}{
			
			\node[d] (poly\Ncorners-\x) at (poly\Ncorners.corner \x){};\node[anchor=\x*(360/\Ncorners)]at(poly\Ncorners.corner \x){$c_{\x}$};
			
		}
		\foreach\X in {1,...,24}{
			\foreach\Y in {1,...,24}{
				\pgfmathtruncatemacro{\Z}{abs(mod(abs(26+\X-\Y),26)-19)}
				\pgfmathtruncatemacro{\D}{abs(mod(abs(\X),2)-1)}
				
				\ifnum\Z=0
				\ifnum\D=0
				\draw (poly\Ncorners-\X) to (poly\Ncorners-\Y);
				\fi
				\fi

				\pgfmathtruncatemacro{\Z}{abs(mod(abs(26+\X-\Y),26)-15)}
				\pgfmathtruncatemacro{\D}{abs(mod(abs(\X),2)-1)}
				
				\ifnum\Z=0
				\ifnum\D=0
				\draw (poly\Ncorners-\X) to (poly\Ncorners-\Y);
				\fi
				\fi
				
				\pgfmathtruncatemacro{\Z}{\X-\Y}
				\ifnum\Z=1
				\draw (poly\Ncorners-\X) to (poly\Ncorners-\Y);
				\fi
			}
		}
		\end{tikzpicture}
		\begin{tikzpicture}
		\pgfmathtruncatemacro{\Ncorners}{26}
		\node[draw, thick,rotate=90,minimum size=5cm,regular polygon, regular polygon sides=26] at (0,0) 
		(poly\Ncorners) {};
		
		\foreach\x in {1,...,\Ncorners}{
			\node[d] (poly\Ncorners-\x) at (poly\Ncorners.corner \x){};
			\node[anchor=\x*(360/\Ncorners)]at(poly\Ncorners.corner \x){$d_{\x}$};
		}
		\foreach\X in {1,...,\Ncorners}{
			\foreach\Y in {1,...,\Ncorners}{
				\pgfmathtruncatemacro{\Z}{abs(mod(abs(\Ncorners+\X-\Y),\Ncorners)-15)}
				\pgfmathtruncatemacro{\D}{abs(mod(abs(\X),2)-1)}
				\ifnum\Z=0
				\ifnum\D=0
				\draw (poly\Ncorners-\X) to (poly\Ncorners-\Y);
				\fi
				\fi
				\pgfmathtruncatemacro{\Z}{abs(mod(abs(\Ncorners+\X-\Y),\Ncorners)-19)}
				\pgfmathtruncatemacro{\D}{abs(mod(abs(\X),2)-1)}
				\ifnum\Z=0
				\ifnum\D=0
				\draw (poly\Ncorners-\X) to (poly\Ncorners-\Y);
				\fi
				\fi
			}
		}
		\end{tikzpicture}
		
		\caption{Extremal graphs $G_3$ (on the left) and $G_4$ (on the right) for $R_{1}^{\mathcal{BIP}}(4,13)=25$  and $R_{1}^{\mathcal{BIP}}(4,14)=27$  respectively.} \label{fig:bipartite13-14}
	\end{center}	
\end{figure}
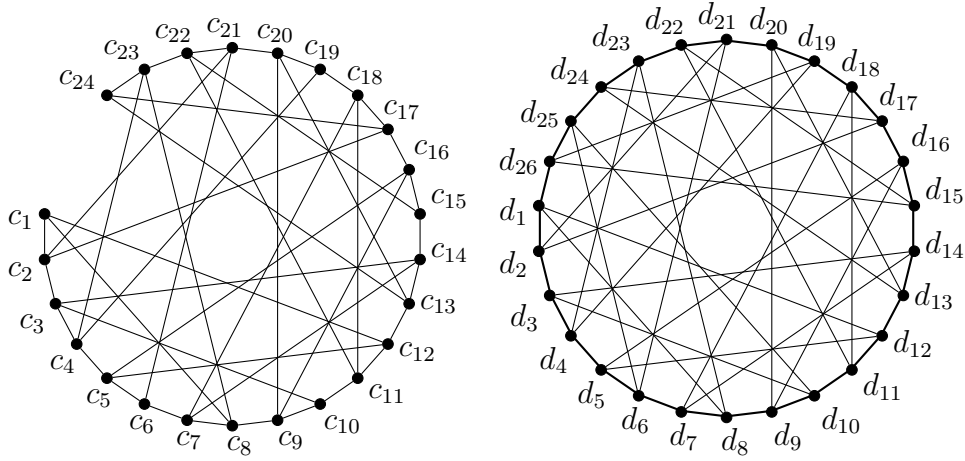
	
\noindent For $j=14$, consider the graph $G_4$ in Figure \ref{fig:bipartite13-14}. Let $O$ and $E$ be the set of vertices whose indices are odd and even, respectively. It can be easily seen that $G_4$ is a 4-regular bipartite graph on $26$ vertices with bipartition $(O,E)$ and it has no $1$-dense $4$-set. We claim that $G_4$ has no $1$-sparse $14$-set. Assume the contrary and take a $1$-sparse $14$-set $S=O_1\cup E_1$ where $O_1\subseteq O$ and $E_1\subseteq E$. Since $G_4$ is symmetric, assume without loss of generality that $|E_1|=x\leq 7$. Note that $|O|=13$ and $|O_1|=14-x$ together give $|O-O_1|=x-1$. Since $S$ is $1$-sparse and each vertex in $G_4$ has degree $4$, each vertex in $E_1$ has at least three neighbors in $O-O_1$. Therefore, there are at least $3x$ edges between $E_1$ and $O-O_1$. Thus, by the pigeonhole principle, there are three vertices $u,v,w\in O-O_1$ such that all four neighbors of each one of $u$, $v$ and $w$ are in $E_1$. Since $x\leq 7$, we get
\begin{eqnarray*}
7 &\geq& |N(u) \cup N(v) \cup N(w)|\\
&\geq& |N(u)|+|N(v)|+|N(w)|-|N(u)\cap N(v)|-|N(u)\cap N(w)|-|N(v)\cap N(w)|\\
&=& 12-|N(u)\cap N(v)|-|N(u)\cap N(w)|-|N(v)\cap N(w)|.
\end{eqnarray*}  
By pigeonhole principle, we have that at least one of $|N(u)\cap N(v)|$, $|N(u)\cap N(w)|$ and $|N(v)\cap N(w)|$ is at least two. Without loss of generality, say $| N(u)\cap N(v)| \geq 2$, which implies that $\{u,v\}\cup (N(u)\cap N(v))$ contains a 1-dense 4-set, a contradiction. As such, we have $R_{1}^{\mathcal{BIP}}(4,14)=27$.\\

\noindent For $j=13$, consider the graph $G_3$ in Figure \ref{fig:bipartite13-14} which is isomorphic to $G_4-\{d_{25},d_{26}\}$. Let $O$ and $E$ be the set of vertices whose indices are odd and even, respectively. We observe that $G_3$ is a bipartite graph on $24$ vertices with bipartition $(O,E)$ and no $1$-dense $4$-set (as it is an induced subgraph of $G_3$). We claim that $G_3$ has no $1$-sparse $13$-set. Assume the contrary and take a $1$-sparse $13$-set $S=O_1\cup E_1$ where $O_1\subseteq O$ and $E_1\subseteq E$. Let us assume $|E_1|=x\leq 6$. The case $|O_1|\leq 6$ directly follows from the following analysis as all the arguments remain valid when $|O_1|\leq 6$ by exchanging the roles of $E_1$ and $O_1$.

Since $|O|=|E|=12$, we have  $|O-O_1|=x-1$. We observe that each of $O$ and $E$ has three vertices of degree three, and the remaining vertices have degree four.  Since $S$ is $1$-sparse and all degrees in $G_3$ are at least three, each vertex in $E_1$ has at least two neighbors in $O-O_1$, implying  $|O-O_1|=x-1\geq 2$, thus $x\geq 3$. It also follows that there are at least $2x$ edges between $E_1$ and $O-O_1$. Then, for $x=3$ and $x=4$,  the number of different pairs of vertices in $O-O_1$ are respectively 1 and 3. This implies that in both cases there are at least two vertices of $E_1$ with two common neighbors in $O-O_1$, yielding a 1-dense 4-set, a contradiction. We now examine the cases $x=5$ and $x=6$ separately: 

	\begin{itemize}
						\item Suppose $x = 5$, in other words say $|O_1|=8$ and $|E-E_1|=7$. Since $S$ is a $1$-sparse set, there are at most $5$ edges between $E_1$ and $O_1$. Therefore, there are at least $(3\cdot 3)+(5\cdot 4)-5=24$ edges between $O_1$ and $E-E_1$. Thus, by pigeonhole principle, there exist three vertices $u,v,w\in E-E_1$ such that each of them has exactly four neighbors in $O_1$. Thus, by using $|O_1|=8$, we get
			\begin{eqnarray*}
				8 &\geq& |N(u) \cup N(v) \cup N(w)|\\
				&\geq& |N(u)|+|N(v)|+|N(w)|-|N(u)\cap N(v)|-|N(u)\cap N(w)|-|N(v)\cap N(w)|\\
				&=& 12-|N(u)\cap N(v)|-|N(u)\cap N(w)|-|N(v)\cap N(w)|
			\end{eqnarray*}  
			Then, by the pigeonhole principle, we have at least one of $|N(u)\cap N(v)|$, $|N(u)\cap N(w)|$ and $|N(v)\cap N(w)|$ is at least two. Without loss of generality, say $|N(u)\cap N(v)|\geq 2$, which implies $u$ and $v$ have at least two common neighbors. However, this is impossible since $G_3$ has no $1$-dense $4$-set.
			
			\item Suppose $x = 6$, in other words say $|O_1|=7$ and $|E-E_1|=6$. Since $S$ is a $1$-sparse set, there are at most $6$ edges between $E_1$ and $O_1$. Therefore, there are at least $(3\cdot 3)+(4\cdot 4)-6=19$ edges between $O_1$ and $E-E_1$. Thus, by the pigeonhole principle, there is at least one vertex $z\in E-E_1$ such that $z$ has exactly four neighbors in $O_1$. Note that we get $|O_1-N(z)|=3$. Let us take $t\in (E-E_1)-\{z\}$. If $t$ is adjacent to all vertices in $O_1-N(z)$, then each vertex in $(E-E_1)-\{z,t\}$ can be adjacent to at most one vertex in each of $N(z)$ and $O_1-N(z)$ since there is no $1$-dense $4$-set in $G_3$. Therefore, there can be at most $(4\cdot 2)+4+4=16$ edges between $O_1$ and $E-E_1$, which is a contradiction. Hence, assume that each vertex in $(E-E_1)-\{z\}$ has at most two neighbors in $O_1-N(z)$ and at most one neighbor in $N(z)$. However, there are at least $15$ edges between $(E-E_1)-\{z\}$ and $O_1$. Then, we note each vertex in $(E-E_1)-\{z\}$ has exactly two neighbors in $O_1-N(z)$ and exactly one neighbor in $N(z)$. Thus, $|O_1-N(z)|=3$ and $|(E-E_1)-\{z\}|=5$ imply there are two vertices in $(E-E_1)-\{z\}$ having two common neighbors in $O_1-N(z)$, which implies $G_3$ has a $1$-dense $4$-set and so we arrive at a contradiction. As a result, we conclude $R_{1}^{\mathcal{BIP}}(4,13)=25$. \qed
	 
		\end{itemize}
	\end{itemize}
\end{proof}

\noindent Using Theorem \ref{thm:bipartitesmall}, we exhaust all cases for $1$-defective Ramsey numbers on bipartite graphs with few exceptions.

\begin{theorem}\label{thm:bipartite1defective}
For all  $j\geq 20$ and $j \in \{15,16,17\}$, we have $R_{1}^{\mathcal{BIP}}(4,j)=2j-1$.
\end{theorem}

\begin{proof}
	Observe that if $j-1\geq 19$ or $j-1 \in \{14,15,16\}$, then there are nonnegative integers $a$, $b$, $c$, $d$ satisfying $j-1=7a+8b+12c+13d$. Hence consider the graph $H$ consisting of the disjoint union of $aG_1$, $bG_2$, $cG_3$, and $dG_4$. Clearly, $|H|=14a+16b+24c+26d=2j-2$ and $H$ has no $1$-dense $4$-set  since none of $G_1,G_2,G_3$ and $G_4$ has a 1-dense 4-set. Moreover, $\alpha_1(H)=a\cdot\alpha_1(G_1)+b\cdot\alpha_1(G_2)+c\cdot\alpha_1(G_3)+d\cdot\alpha_1(G_4)=7a+8b+12c+13d = j-1$, thus $H$ has no $1$-sparse $j$-set. Therefore, $R_{1}^{\mathcal{BIP}}(4,j)\geq2j-1$ for $j\geq 20$ or  $j \in \{15,16,17\}$ and the desired result follows by Theorem \ref{thm:bipartitegeneralRamsey}. \qed \\
\end{proof}

\noindent We claim that the same formula as in Theorem \ref{thm:bipartite1defective} is valid for the remaining cases of $1$-defective Ramsey numbers in bipartite graphs, which we leave as open question:
 
\begin{conjecture}\label{con:bipartite1defectivtherest}
$R_{1}^{\mathcal{BIP}}(4,j)=2j-1$ holds for $j\in \{10,11,12,18,19\}$.
\end{conjecture}

\noindent For $k\geq 2$, we will give the formula for sufficiently large $i$ values by using the fact that bipartite graphs have no relatively dense subsets.

\begin{theorem}\label{thm:bipartitelargevalues}
For all $i\geq 2k+3$ and $k\geq2$, we have $$R_{k}^{\mathcal{BIP}}(i,j)=\begin{cases}
2j-1-k, & \text{if } k+2\leq j\leq 2k\\
2j-1, & \text{if } j\geq 2k+1
\end{cases}$$
\end{theorem}

\begin{proof}
Suppose $i\geq 2k+3$ and $k\geq2$. We claim any bipartite graph has no $k$-dense $i$-set. Assume the contrary, let $D$ be a $k$-dense set with $|D|=i$ in a bipartite graph. Consider two adjacent vertices $u$ and $v$ in $D$. Since $D$ is $k$-dense, we get $d_D(u),d_D(v)\geq i-k-1$. Then, $d_D(u)+d_D(v)\geq 2(i-k-1)=2i-2k-2\geq i+1$ implies $u$ and $v$ have a common neighbor. Thus, $D$ contains a triangle, which is a contradiction. Thus, $R_{k}^{\mathcal{BIP}}(i,j)$ is the smallest integer $n$ such that any bipartite graph on $n$ vertices has a $k$-sparse $j$-set.\\

\noindent If $j\geq 2k+1$, then $K_{j-1,j-1}$ has no $k$-sparse $j$-set and the result follows from the Remark \ref{rem:defectiveinequalityfork} and Theorem \ref{thm:bipartitegeneralRamsey}. If $k+2\leq j\leq 2k$, $K_{j-k-1,j-1}$ has no $k$-sparse $j$-set and we will complete the proof by showing that any bipartite graph on $2j-1-k$ vertices has a $k$-sparse $j$-set.\\

\noindent Let $G$ be a bipartite graph with bipartition $(A,B)$ on $2j-1-k$ vertices. If $|A|$ or $|B|$ at least $j$, we are done. Therefore, we can assume $|A|,|B|\leq j-1$, which implies $|A|,|B|\geq j-k$. Moreover, if both of $|A|$ and $|B|$ are at most $k$, then $G$ must be $k$-sparse and the desired result follows. Hence, without loss of generality, suppose $|A|\geq k+1$. Now, we can find a $k$-sparse $j$-set by taking $k$ vertices from $A$ and $j-k$ vertices from $B$ since $j-k\leq k$. \qed \\
\end{proof}

\noindent We conclude this section by noting that the computation of $R_{k}^{\mathcal{BIP}}(i,j)$ is also open for $k\geq 2$, and $2k+2\geq i\geq k+3$ and $j\geq k+2$ in addition to five specific values in $1$-defective case. 


\section{Split Graphs}

\noindent A graph $G$ is said to be \textit{split} if its vertex set can be partitioned into a clique $K$ and an independent set $I$. In this case, we denote $G=(K,I)$ where $(K,I)$ is called a \textit{split partition}. 
 Let us denote by $\mathcal{SP}$ the class of split graphs. We first recall the classical Ramsey numbers for split graphs.
 
\begin{theorem}\label{thm:splitgeneralRamsey}
\cite{ramseygraphclassesHeggernes}  For all $i,j\geq 3$, we have $R_0^{\mathcal{SP}}(i,j)=i+j-1$.
\end{theorem}

\noindent Since the class of split graphs is closed with respect to taking complements and an empty graph is split, the following remark is a direct consequence of Remark \ref{rem:defectivecomplement} and Lemma \ref{lem:generalkplustwo}.

\begin{remark}\label{rem:splitkplustwo}
	For all $j\geq k+2$, we have $R_k^{\mathcal{SP}}(k+2,j)=R_k^{\mathcal{SP}}(j,k+2)=j$.
\end{remark}


\noindent With the following theorem, we will show that $R_k^{\mathcal{SP}}(i,j)$ is independent from $k$ when a technical assumption on $i$, $j$, $k$ holds.

\begin{theorem}\label{thm:splitsufficientlylarge}
For all $i,j,k$ such that $(i-k-2)(j-k-2)\geq(k+1)^2$, we have $R_k^{\mathcal{SP}}(i,j)=i+j-1$.
\end{theorem}

\begin{proof}
From the inequality $(i-k-2)(j-k-2)\geq(k+1)^2$, we get $i,j\neq k+2$. It also implies $j-1\geq\dfrac{(i-1)(k+1)}{i-k-2}$. Firstly, we have $R_0^{\mathcal{SP}}(i,j)=i+j-1$ for $i,j\geq 3$ from Theorem \ref{thm:splitgeneralRamsey}. Since $R_k^{\mathcal{SP}}(i,j)\leq R_0^{\mathcal{SP}}(i,j)$ from Remark \ref{rem:defectiveinequalityfork} for any $k\geq 1$, we get $R_k^{\mathcal{SP}}(i,j)\leq i+j-1$. Now, we will construct a split graph on $i+j-2$ vertices which has no $k$-dense $i$-set and no $k$-sparse $j$-set. Let $H$ be a split graph on $i+j-2$ vertices with a partition of its vertex set into a clique $K=\{a_1,a_2,...,a_{i-1}\}$ and an independent set $I=\{b_1,b_2,...,b_{j-1}\}$ such that  $a_sb_{(s-1)(k+1)+t} \in E$ for all $s\in \{1,2,...,i-1\}$ and $t\in\{1,2,...,k+1\}$ where indices of the vertices are modulo $j-1$ with the convention $b_0=b_{j-1}$.
	
	\noindent Observe that we have $d(a_s)=k+1$ for all $s\in \{1,2,...,i-1\}$ and $|d(b_{q_1})-d(b_{q_2})|\leq 1$ for all $q_1,q_2\in\{1,2,...,j-1\}$. Since there are exactly $(i-1)(k+1)$ edges between $K$ and $I$, we observe $d(b_q)\in\Big\{\big\lfloor\dfrac{(i-1)(k+1)}{j-1}\big\rfloor,\big\lceil\dfrac{(i-1)(k+1)}{j-1}\big\rceil\Big\}$ for all $q\in\{1,2,...,j-1\}$.\\
	
	\noindent Take a subset of vertices  $J$ in $H$ with $|J|=j$. Since $|I|=j-1$, we note $|J\cap I|\leq j-1$, hence we can say $|J\cap I|=j-l$ and $|J\cap K|=l$ for some $l\geq 1$. Take a vertex  $v\in J\cap K$. Since $K$ is a clique, $v$ is adjacent to $l-1$ vertices in $J\cap K$. On the other hand, $v$ has exactly $k+1$ neighbors in $I$, so it has at least $\max\big\{(k+1)-(l-1),0\big\}$ neighbors in $J\cap I$. Thus, $v$ has at least $(l-1)+\max\big\{(k+1)-(l-1),0\big\}\geq k+1$ neighbors in $J$, implying that $J$ is not $k$-sparse. Thus, $H$ has no $k$-sparse $j$-set.\\
	
	\noindent Take a subset of vertices $C$ in $H$ with $|C|=i$. Since $|K|=i-1$, we note $|C\cap K|\leq i-1$. Hence say $|C\cap K|=i-r$ and $|C\cap I|=r$ for some $r\geq 1$. Take a vertex $u\in C\cap I$. Note that $u$ is adjacent to at most $\big\lceil\dfrac{(i-1)(k+1)}{j-1}\big\rceil$ vertices in $K$. Since $j-1\geq \dfrac{(i-1)(k+1)}{i-k-2}$, we can write $i-k-2\geq \dfrac{(i-1)(k+1)}{j-1}$ and this implies  
	$i-k-2\geq \big\lceil\dfrac{(i-1)(k+1)}{j-1}\big\rceil$ since $i-k-2$ is an integer. As a result, any vertex $u\in C\cap I$ is adjacent to at most $i-k-2$ vertices in $K$. Therefore, we conclude $d_{C\cap K}(u)\leq i-k-2$. Moreover, $u$ is not adjacent to any vertex in $C\cap I$ since $I$ is an independent set. Thus, $u$ misses at least $k+1$ vertices in $C$, implying that $C$ is not $k$-dense. Thus, $H$ has no $k$-dense $i$-set.\\
	
	\noindent It follows that there exists a split graph on $i+j-2$ vertices which has no $k$-dense $i$-set and no $k$-sparse $j$-set, namely $H$. This gives $R_k^{\mathcal{SP}}(i,j)\geq i+j-1$ and the desired result follows. \qed \\
	
\end{proof}

\noindent As a direct corollary of Theorem \ref{thm:splitsufficientlylarge}, we know the exact values of all defective Ramsey numbers for all $i$ and $j$ which are sufficiently large with respect to $k$.

\begin{corollary}\label{cor:splitsufficientlylarge}
	For all $i,j\geq 2k+3$, we have $R_k^{\mathcal{SP}}(i,j)=i+j-1$.
\end{corollary}

\begin{proof}
Since $i,j\geq 2k+3$, we have $i-k-2\geq k+1$ and $i-k-2\geq k+1$, thus we get $(i-k-2)(j-k-2)\geq (k+1)^2$. Then, the result follows from Theorem \ref{thm:splitsufficientlylarge}. \qed \\
\end{proof}

\noindent We now derive the formula for diagonal defective Ramsey numbers in split graphs which are not implied by Corollary \ref{cor:splitsufficientlylarge}.

\begin{theorem}\label{thm:splitdiagonal}
	For all $i$ and $k$ such that $k+3\leq i\leq 2k+2$, we have $R_k^{\mathcal{SP}}(i,i)=3i-2k-4$.
\end{theorem}

\begin{proof}
	Assume $k+3\leq i\leq 2k+2$. Firstly, we will construct a split graph on $3i-2k-5$ vertices which has no $k$-dense $i$-set and no $k$-sparse $i$-set. Consider the graph $G_1$ with split partition  $(A\cup B\cup C,D\cup E)$ with disjoint sets $A,B,C,D,E$ such that $|A|=|B|=i-k-2$, $|C|= 2k+3-i$ and $|D|=|E|=i-k-2$ where $A$ is complete to $D$ and $B$ is complete to $E$.
	
	\noindent 
 We claim $G_1$ has no $k$-dense $i$-set and $k$-sparse $i$-set. Take a subset $S$ of vertices in $G_1$ with $|S|=i$. Observe that $|A\cup B\cup C|=i-1=|C\cup D\cup E|$. Hence there exist $x,y\in S$ such that $x\in A\cup B$ and $y\in C\cup D$. Since $x\in A\cup B$, $x$ misses exactly $i-k-2$ vertices in $G_1$. Hence it misses at most $i-k-2$ vertices in $S$. Therefore, $d_S(x)\geq (i-1)-(i-k-2)=k+1$. Similarly, $y\in D\cup E$ implies $y$ is adjacent to exactly $i-k-2$ vertices in $G_1$, hence it has at most $i-k-2$ neighbors in $S$. Therefore, $y$ misses at least $(i-1)-(i-k-2)=k+1$ vertices in $S$. As a result, $S$ has a vertex that is adjacent to at least $k+1$ other vertices in $S$, namely $x$, and $S$ has a vertex that misses at least $k+1$ other vertices in $S$, namely $y$. Therefore, $S$ is neither $k$-dense nor $k$-sparse. Thus, $G_1$ has no $k$-dense $i$-set, nor $k$-sparse $i$-set, which implies $R_k^{\mathcal{SP}}(i,i)\geq3i-2k-4$.\\
	
	\noindent Let $G$ be a split graph with split partition $(K,I)$ on $3i-2k-4$ vertices. We will show that $G$ has either a $k$-dense $i$-set or a $k$-sparse $i$-set. Assume it has no $k$-sparse $i$-set, we will prove that it contains a $k$-dense $i$-set. Firstly, since $G$ has no $k$-sparse $i$-set, we have $|I|\leq i-1$ and then we get $1\leq i-|I|=|K|-2(i-k-2)$ because $|K|+|I|=3i-2k-4$. Let $S\subseteq K$ be the set of vertices which miss at most $i-k-2$ vertices in $I$. Assume that $|S|\leq 2(i-k-2)$, then we have $|K|-|S|\geq i-|I|$. Hence, consider a subgraph $D$ of $G$ containing exactly $i-|I|$ vertices from $K\backslash S$ and all vertices of $I$. Now, any vertex $u\in K\cap D$ misses at least $i-k-1$ vertices in $I$, thus we have $d_D(u)\leq k$. Moreover, for any $v\in I$, if $v$ is not an isolated vertex in $D$, take a neighbor $w\in K\cap D$ of $v$. Since $K$ is a clique and $I$ is an independent set, we get $d_D(v)\leq d_D(w)\leq k$. As a result, for any vertex $x\in D$, we have $d_D(x)\leq k$, which implies $D$ is a $k$-sparse $i$-set, a contradiction. Hence we have $|S|\geq 2(i-k-2)+1$, and let us take a subset $A\subseteq S$ with $|A|=2(i-k-2)+1$. Let $B\subseteq I$ be the set of vertices which miss more than $i-k-2$ vertices in $A$. Since each vertex in $A$ misses at most $i-k-2$ vertices in $I$, we get $|B|\cdot (i-k-2)< |A|\cdot (i-k-2)$ and so $|B|\leq |A|-1=2(i-k-2)$. This implies $|K\cup(I\backslash B)|\geq (3i-2k-4)-(2i-2k-4)=i$. Observe that: 
	\begin{itemize}
		\item[(i)] Any vertex $y_1$ in $A$ misses at most $i-k-2$ vertices in $I$. Since $i\leq 2k+2$, $y_1$ misses at most $k$ vertices in $I$.
		\item[(ii)] Any vertex $y_2$ in $I\backslash B$ misses at most $i-k-2$ vertices in $A$. Since $i\leq 2k+2$, $y_2$ misses at most $k$ vertices in $A$. On the other hand, since $|A|= 2(i-k-2)+1$, $y_2$ is adjacent to at least $i-k-1$ vertices in $A$.
		\item[(iii)] Any vertex $y_3$ in $K\backslash A$ is adjacent to all vertices in $A$ since $K$ is a clique. Hence, $y_3$ has at least $i-k-1$ neighbors in $A$ since $|A|=2(i-k-2)+1$ and $i\geq k+3$. 
	\end{itemize}
	
	\noindent Now, if $|A\cup(I\backslash B)|\geq i$, since $|A|< i$ we get $|I\backslash B|\geq 2k+3-i$. Hence we can choose $|A\cup(I\backslash B)|-i$ elements from $I\backslash B$, denote the set of these elements by $I_1$, and let us examine the set $A\cup\big(I\backslash (B\cup I_1)\big)$ of size $i$. From observation (i), any vertex $z_1$ in $A$ misses at most $k$ elements from $I\backslash (B\cup I_1)$, and $z_1$ is adjacent to all other vertices in $A$. On the other hand, any vertex $z_2$ in $I\backslash (B\cup I_1)$ is adjacent to at least $i-k-1$ vertices in $A$ from observation (ii), hence $z_2$ misses at most $(i-1)-(i-k-1)=k$ vertices in $A\cup\big(I\backslash (B\cup I_1)\big)$. As a result, $A\cup\big(I\backslash (B\cup I_1)\big)$ forms a $k$-dense $i$-set. 

If, however, $|A\cup(I\backslash B)|\leq i-1$, then we can choose $i-|A\cup(I\backslash B)|$ elements from $K\backslash A$ since $|K\cup(I\backslash B)|\geq i$; denote the set of these elements by $A_1$, and let us show that the set $(A\cup A_1)\cup(I\backslash B)$ of size $i$ is $k$-dense. Take $z_3\in A$ and $z_4\in A_1$, then $z_3$ misses at most $k$ elements in $I\backslash B$ by observation (i), and $z_4$ is adjacent to at least $i-k-1$ vertices in $(A\cup A_1)\cup(I\backslash B)$ by observation (iii). Hence each one of $z_3$ and $z_4$ can miss at most $k$ vertices in $(A\cup A_1)\cup(I\backslash B)$. On the other hand, for any $z_5\in I\backslash B$, $z_5$ is adjacent to at least $i-k-1$ vertices in $A$ from observation (ii), hence it can miss at most $(i-1)-(i-k-1)=k$ vertices in $(A\cup A_1)\cup(I\backslash B)$. As a result, $(A\cup A_1)\cup(I\backslash B)$ forms a $k$-dense $i$-set. 

We conclude that if $G$ has no $k$-sparse $i$-set, then it has a $k$-dense $i$-set, and the result follows. \qed \\
\end{proof}


\noindent We will end this section by completing the list of $1$-defective and $2$-defective Ramsey numbers in split graphs. By Remark \ref{rem:splitkplustwo}, Corollary \ref{cor:splitsufficientlylarge}, and Theorem \ref{thm:splitdiagonal}, the only open cases are $R_{1}^{\mathcal{SP}}(4,5)$ and $R_{1}^{\mathcal{SP}}(4,6)$ for $1$-defective Ramsey numbers, and $R_{2}^{\mathcal{SP}}(6,7)$, $R_{2}^{\mathcal{SP}}(6,8)$ and $R_{2}^{\mathcal{SP}}(5,j)$ for $6\leq j\leq 12$, for $2$-defective Ramsey numbers.

\begin{theorem}\label{thm:split1and2defective}
With the preceding notation,
\begin{itemize}
\item[(i)] $R_{1}^{\mathcal{SP}}(4,5)=7$ and $R_{1}^{\mathcal{SP}}(4,6)=8$.
\item[(ii)] $R_{2}^{\mathcal{SP}}(6,7)=11$ and $R_{2}^{\mathcal{SP}}(6,8)=12$.
\item[(iii)] $R_{2}^{\mathcal{SP}}(5,6)=8$ and $R_{2}^{\mathcal{SP}}(5,7)=9$.
\item[(iv)] $R_{2}^{\mathcal{SP}}(5,j)=j+3$ for all $8\leq j\leq 12$.
\end{itemize}
\end{theorem}

\begin{proof} 

\begin{itemize}
\item[(i)]  Observe that the graph $S_1$ (resp. $S_2$) in Figure \ref{fig:split_I} has no $1$-dense $4$-set and no $1$-sparse $5$-set (resp. no $1$-sparse $6$-set). Since $|S_1|=6$  and $|S_2|=7$, this gives $R_{1}^{\mathcal{SP}}(4,5)\geq 7$ and $R_{1}^{\mathcal{SP}}(4,6)\geq 8$. Let $G=(K,I)$ be a split graph on $7$ (resp. $8$) vertices. We claim $G$ has either a $1$-dense $4$-set or a $1$-sparse $5$-set (resp. $1$-sparse $6$-set). If $|K|\geq 4$ or $|I|\geq 5$ (resp. $|I|\geq 6$), we are done. So, since $|G|=7$ (resp. $|G|=8$), assume $|K|=3$ and $|I|= 4$ (resp. $|I|= 5$). If a vertex $u\in K$ has at most one neighbor in $I$, then $\{u\}\cup I$ is a $1$-sparse $5$-set (resp. $1$-sparse $6$-set) and we are done. So, assume each one of the three vertices in $K$ has at least two neighbors in $I$, which implies there are at least 6 edges between $K$ and $I$. Since $|I|<6$, by pigeonhole principle, there exists $v\in I$ such that $v$ has at least two neighbors in $K$. Thus, $\{v\}\cup K$ is a $1$-dense $4$-set and we are done. As a result, $R_{1}^{\mathcal{SP}}(4,5)=7$ (resp. $R_{1}^{\mathcal{SP}}(4,6)= 8$). 

\def\r{4pt}
\def\dy{1cm}
\tikzset{c/.style={draw,circle,fill=white,minimum size=\r,inner sep=0pt,
		anchor=center},
	d/.style={draw,circle,fill=black,minimum size=\r,inner sep=0pt, anchor=center}}

\begin{figure}[h]
	\begin{center}
		\begin{tikzpicture}
		\node[thick,minimum size=3cm] at (-0.3,2) {$x_1$};
		\node[d] at (0,2) {};
		\node[thick,minimum size=3cm]  at (-0.3,1) {$x_2$};
		\node[d] at (0,1) {};
		\node[thick,minimum size=3cm] at (1.3,3) {$y_1$};
		\node[d] at (1,3) {};
		\node[thick,minimum size=3cm] at (1.3,2) {$y_2$};
		\node[d] at (1,2) {};
		\node[thick,minimum size=3cm] at (1.3,1) {$y_3$};
		\node[d] at (1,1) {};
		\node[thick,minimum size=3cm] at (1.3,0) {$y_4$};
		\node[d] at (1,0) {};
		\draw (0,2) to (1,3)
		(0,2) to (1,2)
		(0,1) to (1,1)
		(0,1) to (1,0)
		(0,1) to (0,2);
		\end{tikzpicture}
		\begin{tikzpicture}
		\node[thick,minimum size=3cm] at (-0.3,2) {$x_1$};
		\node[d] at (0,2) {};
		\node[thick,minimum size=3cm]  at (-0.3,1) {$x_2$};
		\node[d] at (0,1) {};
		\node[thick,minimum size=3cm]  at (-0.8,1.5) {$x_3$};
		\node[d] at (-0.5,1.5) {};
		\node[thick,minimum size=3cm] at (1.3,3) {$y_1$};
		\node[d] at (1,3) {};
		\node[thick,minimum size=3cm] at (1.3,2) {$y_2$};
		\node[d] at (1,2) {};
		\node[thick,minimum size=3cm] at (1.3,1) {$y_3$};
		\node[d] at (1,1) {};
		\node[thick,minimum size=3cm] at (1.3,0) {$y_4$};
		\node[d] at (1,0) {};
		
		\draw (-0.5,1.5) to (0,2)
		(-0.5,1.5) to (0,1)
		(0,2) to (1,3)
		(0,2) to (1,2)
		(0,1) to (1,1)
		(0,1) to (1,0)
		(0,1) to (0,2);
		\end{tikzpicture}
\vspace{-1cm}
		\caption{Extremal graphs $S_1$ (on the left) and $S_2$ (on the right) for  $R_{1}^{\mathcal{SP}}(4,5)$ and $R_{1}^{\mathcal{SP}}(4,6)$ respectively.} \label{fig:split_I}
	\end{center}
\end{figure}
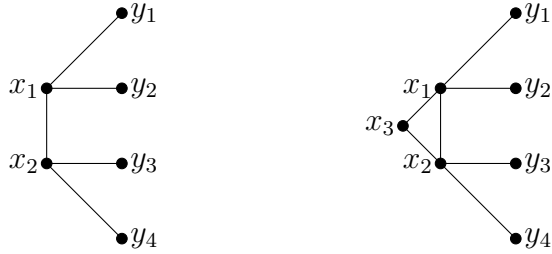

\item[(ii)] It can be seen that the graph $S_3$ (resp. $S_4$) in Figure \ref{fig:split_II} has no $2$-dense $6$-set and no $2$-sparse $7$-set (resp. no $2$-sparse $8$-set). Since $|S_3|=10$ and $|S_4|=11$, we have $R_{2}^{\mathcal{SP}}(6,7)\geq 11$ and $R_{2}^{\mathcal{SP}}(6,8)\geq 12$. Let $G=(K,I)$ be a split graph on $11$ (resp. $12$) vertices. We claim $G$ has either a $2$-dense $6$-set or a $2$-sparse $7$-set (resp. $2$-sparse $8$-set). If $|K|\geq 6$ or $|I|\geq 7$ (resp. $|I|\geq 8$), we are done. So, since $|G|=11$ (resp. $|G|=12$), assume $|C|= 5$ and $|I| = 6$ (resp. $|I|=7$). If a vertex $u\in C$ has at most two neighbors in $I$, then $\{u\}\cup I$ is a $2$-sparse $7$-set (resp. $2$-sparse $8$-set) and we are done. So, assume each one of the five vertices in $K$ has at least three neighbors in $I$, which implies there are at least 15 edges between $K$ and $I$. Since $2|I|<15$, by the pigeonhole principle, there exists $v\in I$ such that $v$ has at least three neighbors in $K$. Thus, $\{v\}\cup K$ is a $2$-dense $6$-set and we are done. As a result, $R_{2}^{\mathcal{SP}}(6,7)=11$ (resp. $R_{2}^{\mathcal{SP}}(6,8)=12$). 

\begin{figure}[h]
	\begin{tikzpicture}
	\node[thick,minimum size=3cm] at (-0.3,1) {$x_1$};
	\node[d] at (0,1) {};
	\node[thick,minimum size=3cm]  at (-0.3,0) {$x_2$};
	\node[d] at (0,0) {};
	\node[thick,minimum size=3cm] at (-1.3,2) {$x_3$};
	\node[d] at (-1,2) {};
	\node[thick,minimum size=3cm] at (-1.3,-1) {$x_4$};
	\node[d] at (-1,-1) {};
	\node[thick,minimum size=3cm] at (1.3,3) {$y_1$};
	\node[d] at (1,3) {};
	\node[thick,minimum size=3cm] at (1.3,2) {$y_2$};
	\node[d] at (1,2) {};
	\node[thick,minimum size=3cm] at (1.3,1) {$y_3$};
	\node[d] at (1,1) {};
	\node[thick,minimum size=3cm] at (1.3,0) {$y_4$};
	\node[d] at (1,0) {};
	\node[thick,minimum size=3cm] at (1.3,-1) {$y_5$};
	\node[d] at (1,-1) {};
	\node[thick,minimum size=3cm] at (1.3,-2) {$y_6$};
	\node[d] at (1,-2) {};
	
	\draw (0,1) to (1,3)
	(0,1) to (1,2)
	(0,1) to (1,1)
	(-1,2) to (1,3)
	(-1,2) to (1,2)
	(-1,2) to (1,1)
	(0,0) to (1,0)
	(0,0) to (1,-1)
	(0,0) to (1,-2)
	(-1,-1) to (1,-2)
	(-1,-1) to (1,-1)
	(-1,-1) to (1,0)
	(0,1) to (-1,2)
	(0,1) to (-1,-1)
	(0,1) to (0,0)
	(-1,2) to (-1,-1)
	(-1,2) to (0,0)
	(0,0) to (-1,-1);
	\end{tikzpicture}
	\hfill
	\begin{tikzpicture}
	\node[thick,minimum size=3cm] at (0.2,0.8) {$x_1$};
	\node[d] at (0,1) {};
	\node[thick,minimum size=3cm]  at (0.2,0.15) {$x_2$};
	\node[d] at (0,0) {};
	\node[thick,minimum size=3cm] at (-1.3,2) {$x_3$};
	\node[d] at (-1,2) {};
	\node[thick,minimum size=3cm] at (-1.3,-1) {$x_4$};
	\node[d] at (-1,-1) {};
	\node[thick,minimum size=3cm] at (-2.3,0.5) {$x_5$};
	\node[d] at (-2,0.5) {};
	\node[thick,minimum size=3cm] at (1.3,3) {$y_1$};
	\node[d] at (1,3) {};
	\node[thick,minimum size=3cm] at (1.3,2) {$y_2$};
	\node[d] at (1,2) {};
	\node[thick,minimum size=3cm] at (1.3,1) {$y_3$};
	\node[d] at (1,1) {};
	\node[thick,minimum size=3cm] at (1.3,0) {$y_4$};
	\node[d] at (1,0) {};
	\node[thick,minimum size=3cm] at (1.3,-1) {$y_5$};
	\node[d] at (1,-1) {};
	\node[thick,minimum size=3cm] at (1.3,-2) {$y_6$};
	\node[d] at (1,-2) {};
	
	\draw (0,1) to (1,3)
	(0,1) to (1,2)
	(0,1) to (1,1)
	(-1,2) to (1,3)
	(-1,2) to (1,2)
	(-1,2) to (1,1)
	(0,0) to (1,0)
	(0,0) to (1,-1)
	(0,0) to (1,-2)
	(-1,-1) to (1,-2)
	(-1,-1) to (1,-1)
	(-1,-1) to (1,0)
	(0,1) to (-1,2)
	(0,1) to (-1,-1)
	(0,1) to (0,0)
	(-1,2) to (-1,-1)
	(-1,2) to (0,0)
	(0,0) to (-1,-1)
	(-2,0.5) to (0,0)
	(-2,0.5) to (-1,2)
	(-2,0.5) to (-1,-1)
	(-2,0.5) to (0,1);
	\end{tikzpicture}
\vspace{-1cm}
	\caption{Extremal graphs $S_3$ (on the left) and $S_4$ (on the right) for  $R_{2}^{\mathcal{SP}}(6,7)$ and $R_{1}^{\mathcal{SP}}(6,8)$ respectively.} \label{fig:split_II}
\end{figure}
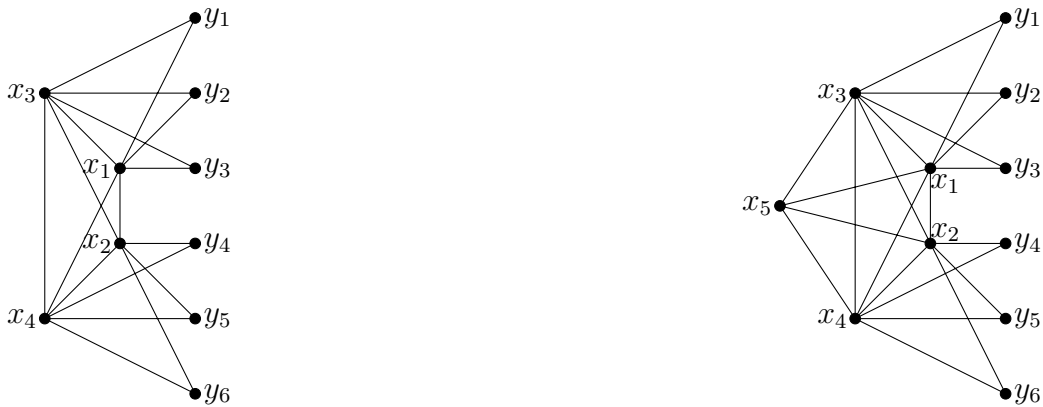

\begin{figure}[h]
	\begin{center}
		
		\begin{subfigure}[b]{0.30\textwidth}
			\centering
			\begin{tikzpicture}
			\node[thick,minimum size=3cm] at (-0.3,2) {$x_1$};
			\node[d] at (0,2) {};
			\node[thick,minimum size=3cm]  at (0.7,1.15) {$x_2$};
			\node[d] at (0.5,1) {};
			\node[thick,minimum size=3cm] at (-0.3,0) {$x_3$};
			\node[d] at (-0,0) {};
			\node[thick,minimum size=3cm]  at (-0.8,1) {$x_4$};
			\node[d] at (-0.5,1) {};
			\node[thick,minimum size=3cm] at (1.8,2) {$y_1$};
			\node[d] at (1.5,2) {};
			\node[thick,minimum size=3cm] at (1.8,1) {$y_2$};
			\node[d] at (1.5,1) {};
			\node[thick,minimum size=3cm] at (1.8,0) {$y_3$};
			\node[d] at (1.5,0) {};
			\draw (0,2) to (1.5,2)
			(0,0) to (1.5,0)
			(0.5,1) to (1.5,1)
			(0,2) to (0,0)
			(0,2) to (-0.5,1)
			(0,2) to (0.5,1)
			(0,0) to (-0.5,1)
			(0,0) to (0.5,1)
			(0.5,1) to (-0.5,1);
			\end{tikzpicture}
		\end{subfigure}
		\begin{subfigure}[b]{0.30\textwidth}
			\centering
			\begin{tikzpicture}
			\node[thick,minimum size=3cm] at (-0.3,2) {$x_1$};
			\node[d] at (0,2) {};
			\node[thick,minimum size=3cm]  at (0.7,1.15) {$x_2$};
			\node[d] at (0.5,1) {};
			\node[thick,minimum size=3cm] at (-0.3,0) {$x_3$};
			\node[d] at (-0,0) {};
			\node[thick,minimum size=3cm]  at (-0.8,1) {$x_4$};
			\node[d] at (-0.5,1) {};
			\node[thick,minimum size=3cm] at (1.8,1.5) {$y_1$};
			\node[d] at (1.5,1.5) {};
			\node[thick,minimum size=3cm] at (1.8,1) {$y_2$};
			\node[d] at (1.5,1) {};
			\node[thick,minimum size=3cm] at (1.8,0) {$y_4$};
			\node[d] at (1.5,0) {};
			\node[thick,minimum size=3cm] at (1.8,0.5) {$y_3$};
			\node[d] at (1.5,0.5) {};
			\draw (0,2) to (1.5,1.5)
			(0,0) to (1.5,0.5)
			(0.5,1) to (1.5,1)
			(0,2) to (0,0)
			(0,2) to (-0.5,1)
			(0,2) to (0.5,1)
			(0,0) to (-0.5,1)
			(0,0) to (0.5,1)
			(0.5,1) to (-0.5,1);
			\end{tikzpicture}
		\end{subfigure}
		\begin{subfigure}[b]{0.30\textwidth}
			\centering
			\begin{tikzpicture}
			\node[thick,minimum size=3cm] at (-0.3,2) {$x_1$};
			\node[d] at (0,2) {};
			\node[thick,minimum size=3cm]  at (0.7,1.2) {$x_2$};
			\node[d] at (0.5,1) {};
			\node[thick,minimum size=3cm] at (-0.3,0) {$x_3$};
			\node[d] at (-0,0) {};
			\node[thick,minimum size=3cm]  at (-0.8,1) {$x_4$};
			\node[d] at (-0.5,1) {};
			\node[thick,minimum size=3cm] at (1.8,2.25) {$y_1$};
			\node[d] at (1.5,2.25) {};
			\node[thick,minimum size=3cm] at (1.8,1.75) {$y_2$};
			\node[d] at (1.5,1.75) {};
			\node[thick,minimum size=3cm] at (1.8,1.15) {$y_3$};
			\node[d] at (1.5,1.15) {};
			\node[thick,minimum size=3cm] at (1.8,0.75) {$y_4$};
			\node[d] at (1.5,0.75) {};
			\node[thick,minimum size=3cm] at (1.8,0.4) {$y_5$};
			\node[d] at (1.5,0.4) {};
			\node[thick,minimum size=3cm] at (1.8,-0) {$y_6$};
			\node[d] at (1.5,-0) {};
			\draw (0,2) to (1.5,2.25)
			(0,2) to (1.5,1.75)
			(0,0) to (1.5,0.4)
			(0,0) to (1.5,-0)
			(0.5,1) to (1.5,1.15)
			(0.5,1) to (1.5,0.75)
			(0,2) to (0,0)
			(0,2) to (-0.5,1)
			(0,2) to (0.5,1)
			(0,0) to (-0.5,1)
			(0,0) to (0.5,1)
			(0.5,1) to (-0.5,1);
			\end{tikzpicture}
		\end{subfigure}
\vspace{-1cm}
		\caption{Extremal graphs $S_5$ (on the left) and $S_6$ (in the middle) for  $R_{2}^{\mathcal{SP}}(5,6)$ and $R_{2}^{\mathcal{SP}}(5,7)$ respectively. The disjoint union of $S_7$ (on the right) with $j-8$ isolated vertices is an extremal graph for $R_{2}^{\mathcal{SP}}(5,j)$  for all $8\leq j\leq 12$.} \label{fig:split_III}
	\end{center}
\end{figure}
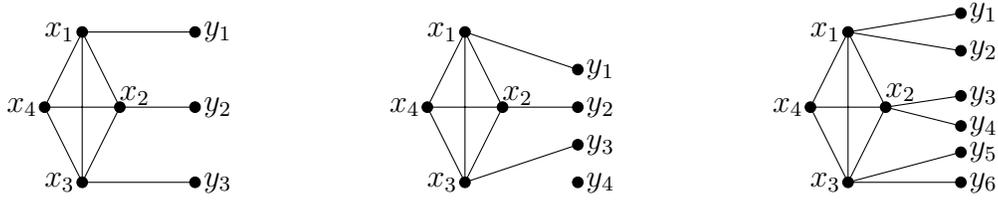

\item[(iii)] Note that the graph $S_5$ (resp. $S_6$) in Figure \ref{fig:split_III} has no $2$-dense $5$-set and no $2$-sparse $6$-set (resp. no $2$-sparse $7$-set). Since $|S_5|=7$ and $|S_6|=8$, this gives $R_{2}^{\mathcal{SP}}(5,6)\geq 8$ and $R_{2}^{\mathcal{SP}}(5,7)\geq 9$. Let $G=(K,I)$ be a split graph on $8$ (resp. $9$) vertices. We claim $G$ has either a $2$-dense $5$-set or a $2$-sparse $6$-set (resp. $2$-sparse $7$-set). If $|K|\geq 5$ or $|I|\geq 6$ (resp. $|I|\geq 7$), we are done. Assume $|K|\leq 4$ and $|I|\leq 5$ (resp. $|I|\leq 6$). Since $|G|=8$ (resp. $|G|=9$), we have two cases:
\begin{itemize}
	\item Suppose $|K|=4$ and $|I|=4$ (resp. $|I|=5$). If a vertex $v\in I$ has at least two neighbors in $K$, then $\{v\}\cup K$ is a $2$-dense $5$-set and we are done. So, assume each vertex in $I$ has at most one neighbor in $K$, hence there are at most $|I|$ edges between $K$ and $I$. Since $|I|<6$, by the pigeonhole principle, there are two vertices $x,y\in K$ such that each of $x$ and $y$ has at most one neighbor in $I$. Thus, $\{x,y\}\cup I$ is a $2$-sparse $6$-set (resp. $2$-sparse $7$-set) and we are done.
	\item Suppose $|K|=3$ and $|I|=5$ (resp. $|I|=6$). If a vertex $u\in K$ has at most two neighbors in $I$, then $\{u\}\cup I$ is a $2$-sparse $6$-set (resp. $2$-sparse $7$-set) and we are done. So, assume each one of the three vertices in $K$ has at least three neighbors in $I$, which implies there are at least $9$ edges between $K$ and $I$. Since $|K|=3$ and $|I|<7$, by the pigeonhole principle, there are two vertices $z,w\in I$ such that each of $z$ and $w$ has at least two neighbors in $K$. Thus, $\{z,w\}\cup K$ is a $2$-dense $5$-set and we are done. 
\end{itemize}
In either case, we conclude that $R_{2}^{\mathcal{SP}}(5,6)=8$ (resp. $R_{2}^{\mathcal{SP}}(5,7)=9$).

\item[(iv)]  It can be seen that the graph $S_7$ in Figure \ref{fig:split_III} has no $2$-dense $5$-set and no $2$-sparse $8$-set, and $|S_7|=10$. Construct the graph $S$ as the disjoint union of $S_7$ and $j-8$ isolated vertices. Clearly, $S$ is a split graph on $j+2$ vertices which has no $2$-dense $5$-set and no $2$-sparse $j$-set, which implies $R_{2}^{\mathcal{SP}}(5,j)\geq j+3$. Let $G=(K,I)$ be a split graph on $j+3$ vertices for some $8\leq j\leq 12$. We claim $G$ has either a $2$-dense $5$-set or a $2$-sparse $j$-set. If $|K|\geq 5$ or $|I|\geq j$, we are done. So, assume $|K|\leq 4$ and $|I|\leq j-1$. Since $|G|=j+3$, we get $|K|=4$ and $|I|=j-1$. If a vertex $u\in K$ has at most two neighbors in $I$, then $\{u\}\cup I$ is a $2$-sparse $j$-set and we are done. Hence, assume each one of the four vertices in $K$ has at least three neighbors in $I$, which implies there are at least 12 edges between $K$ and $I$. Since $|I|<12$, by the pigeonhole principle, there exists $v\in I$ such that $v$ has at least two neighbors in $K$. Thus, $\{v\}\cup K$ is a $2$-dense $5$-set and we are done. \qed
\end{itemize}

\end{proof}

\noindent Finally, in light of the above results, we conjecture the following formula for all defective Ramsey numbers in split graphs.

\begin{conjecture}\label{conj:splitmostgeneral}
	For all $i,j\geq k+2$,  $$R_k^{\mathcal{SP}}(i,j)=i+j-1-\max\Bigg\{0,\Big\lceil\dfrac{(k+1)^2-(i-k-2)(j-k-2)}{\min\{i,j\}}\Big\rceil\Bigg\}.$$
\end{conjecture}


\section{Cographs}\label{sec:cographs}

\textit{Cographs} are the graphs that can be generated from a single vertex by taking complements and disjoint unions of two cographs. Let $P_4$ be a path on four vertices. It is known that a graph is a cograph if and only if it does not contain $P_4$ as an induced subgraph \cite{complementcograph}. Since the complement of a $P_4$ is again a $P_4$, cographs form a self-complementary graph class.  Alternatively, a graph $G$ is a cograph if and only if the complement of every connected subgraph of $G$ is disconnected. Throughout this paper, we will use the latter definition of cographs.  Let $\mathcal{CO}$ denote the class of cographs. The classical Ramsey numbers in cographs are given by the following formula in \cite{ramseygraphclassesHeggernes}. 
\begin{theorem}\cite{ramseygraphclassesHeggernes} \label{thm:cographgeneralRamsey}
$R_{0}^{\mathcal{CO}}(i,j)=1+(i-1)(j-1)$ for all $i,j\geq 3$.
\end{theorem}

\begin{remark}\label{rem:cographkplustwo}
$R_{k}^{\mathcal{CO}}(k+2,j)=R_{k}^{\mathcal{CO}}(j,k+2)=j$ for all $j\geq k+2$.
\end{remark}

\begin{proof}
Since the complement of a cograph is also a cograph, by Remark \ref{rem:defectivecomplement}, we only need to prove that $R_k^{\mathcal{CO}}(k+2,j)=j$ for $j\geq k+2$. Moreover, since an empty graph is a cograph, the result follows from Lemma \ref{lem:generalkplustwo}. \qed
\end{proof}

\begin{lemma}\label{lem:cographconnected}
Let $1\leq s\leq k$ and $G$ be a connected cograph on at least $k+s+2$ vertices which has no $k$-dense $(k+s+2)$-set. Then, there exists a subgraph $L$ of $G$ with $1\leq |L|\leq s$ such that each vertex of $L$ is adjacent to all vertices of $G-L$. 
\end{lemma}

\begin{proof}
Note that $\overline{G}$ is a disconnected cograph which has no $k$-sparse $(k+s+2)$-set. Let $L$ be the union of connected components of $\overline{G}$ each one of whose size is at most $k$. Observe that each vertex of $L$ is adjacent to all vertices of $G-L$ in $G$. We claim $1\leq |L|\leq s$. We note that $L$ forms a $k$-sparse set in $\overline{G}$, thus we have $|L|\leq k+s+1$ since $\overline{G}$ has no $k$-sparse $(k+s+2)$-set. This implies $|\overline{G}-L|\geq 1$, in other words, by definition of $L$, $\overline{G}$ has at least one connected component of size at least $k+1$. If there are two connected components in $\overline{G}$ each one of size at least $k+1$, say $C_1$ and $C_2$, then by taking $k+1$ vertices from each of $C_1$ and $C_2$ we would get a $k$-sparse $(2k+2)$-set (by Remark \ref{rem:sparseunion}) where $2k+2 \geq k+s+2$ since $s\leq k$, a contradiction. Therefore, $\overline{G}-L$ has a unique connected component, implying $|L|\geq 1$ since $\overline{G}$ is disconnected. Finally, if $|L|\geq s+1$, we would get a $k$-sparse $(k+s+2)$-set by taking $s+1$ elements from $L$ and $k+1$ elements from $\overline{G}-L$ from Remark \ref{rem:sparseunion}. As a result, we have $1\leq |L|\leq s$. \qed
\end{proof}

\begin{lemma}\label{lem:cographsmalliandj}
Let $k+2\leq i,j\leq 2k+2$. Then, $R_{k}^{\mathcal{CO}}(i,j)=i+j-k-2$.
\end{lemma}

\begin{proof}
We will prove this by strong induction on $i+j$. Firstly, if $i=k+2$ or $j=k+2$  the lemma holds by Remark \ref{rem:cographkplustwo}. Hence assume the lemma holds for $i,j \geq k+2$ and $i+j<t$ for some $t\geq 2k+6$, and take $i,j\geq k+3$ with $i+j=t$. 

Consider the graph $H$ obtained by the join of a clique $K$ of size $i-k-2$ and an independent set $J$ of size $j-1$. Clearly, $H$ is a cograph on $i+j-k-3$ vertices. Note that $H$ has no $k$-dense $i$-set since any set of size $i$ has to contain at least $k+2$ vertices from $J$. Also, $H$ has no $k$-sparse $j$-set since any set of size $j$ has to contain a vertex from $K$, which is adjacent to all the other $i-2\geq k+1$ vertices in the set. Therefore, we have $R_{k}^{\mathcal{CO}}(i,j)\geq i+j-k-2$. 

Now, take a cograph $G$ on $i+j-k-2$ vertices. We will prove that $G$ has either a $k$-dense $i$-set or $k$-sparse $j$-set. If every connected component of $G$ has size at most $k$, then clearly $G$ is $k$-sparse, implying $\alpha_k(G)=|G|=i+j-k-2\geq j$ and we are done. If $G$ has at least two connected component of size at least $k+1$, say $C_1$ and $C_2$, then we can take $k+1$ vertices from each of $C_1$ and $C_2$, which implies $G$ has a $k$-sparse $(2k+2)$-set. Since $2k+2\geq j$, the result follows. Therefore, we can assume $G$ has a unique connected component of size at least $k+1$, say $C$. Note that we have $|G-C|=\alpha_k(G-C)$ and so $\alpha_k(G)=\alpha_k(C)+|G-C|\geq k+1+|G-C|$ from Remark \ref{rem:ksparseunion}. Therefore, if $|G-C|\geq j-k-1$ then $G$ has a $k$-sparse $j$-set and we are done. Thus, assume $|G-C|\leq j-k-2$. Since $|G|=i+j-k-2$, this implies $|C|\geq i$. Now, if $C$ has a $k$-dense $i$-set, we are done. Otherwise, from Lemma \ref{lem:cographconnected}, there exists a subgraph $L$ of $C$ with $1\leq |L|\leq i-k-2$ such that each vertex of $L$ is adjacent to all vertices of $C-L$. Now, since $C$ has no $k$-dense $i$-set, $C-L$ has no $k$-dense $(i-|L|)$-set. Therefore, we have $|C-L|\leq R_{k}^{\mathcal{CO}}(i-|L|,\alpha_k(C)+1)-1$. If $\alpha_k(C)\geq j$ we are done, so assume $\alpha_k(C)\leq j-1$. We can verify as follows that the induction hypothesis is valid for $R_{k}^{\mathcal{CO}}(i-|L|,\alpha_k(C)+1)$. By Lemma \ref{lem:cographconnected}, we have $k+2\leq i-|L|$ and since $i\leq 2k+1$ and $|L|\geq 1$, we have $i-|L| \leq 2k+2$.  We also have $k+2\leq 1+\alpha_k(C)\leq 1+(j-1)\leq 2k+2$ and $i-|L|+\alpha_k(C)+1\leq i+j-1$ since $|L|\geq 1$. From the induction hypothesis, we have $R_{k}^{\mathcal{CO}}(i-|L|,\alpha_k(C)+1)=i-|L|+\alpha_k(C)+1-k-2$. Hence we get $|C|\leq |L|+i-|L|+\alpha_k(C)+1-k-2-1=i-k-2+\alpha_k(C)$. Then, by using Remark \ref{rem:ksparseunion}, we obtain $i+j-k-2=|G|=|C|+|G-C|\leq i-k-2+\alpha_k(C)+\alpha_k(G-C)=i-k-2+\alpha_k(G)$. This implies $\alpha_k(G)\geq j$ and we are done. \qed 
\end{proof}

\begin{lemma}\label{lem:cographmoduloremainder}
Let $\{x\}$ denote the value of the integer $x$ modulo $m$ for some $m\geq 2$. Hence, we have $\{x\}-\{y\}\leq \{x-y\}$.
\end{lemma}

\begin{proof}
Write $x=ma+s$ and $y=mb+t$ for some $0\leq s,t< m$. Then, we have the equality $x-y=m(a-b)+(s-t)$ and so $$\{x-y\}=\begin{cases}
s-t, & \text{ if }s\geq t\\
m+(s-t), & \text{ if } s< t
\end{cases}$$
Since $\{x\}-\{y\}=s-t$, we get $$\{x-y\}-\Big(\{x\}-\{y\}\Big)=\begin{cases}
0, & \text{ if }s\geq t\\
m, & \text{ if }s< t
\end{cases} \,\text{ and so the result follows. }$$ \qed
\end{proof}

\begin{lemma}\label{lem:cographmoduloremainder2}
Let $\{x\}$ denote the value of the integer $x$ modulo $m$ for some $m\geq 2$. Then, $1+\dfrac{ab-\{a\}\{b\}}{m}\geq 1+b$ for all $a,b\in\mathbb{N}$ provided that $a\geq m$.
\end{lemma}

\begin{proof}
For any $t\in\mathbb{N}$ with $t-m\geq 0$, we have the following by using the facts $t\geq \{t\}$ and $\{t\}=\{t-m\}$.
\begin{eqnarray*}
	1+\dfrac{ab-\{a\}\{b\}}{m}=1+b+\dfrac{(a-m)b-\{a\}\{b\}}{m}
	&=&1+b+\dfrac{(a-m)b-\{a-m\}\{b\}}{m}\\
	&\geq& 1+b\,\, \text{ and we are done.} 
\end{eqnarray*}\qed
\end{proof}

The following will be used in Theorem \ref{thm:cographkdefective} to show the existence of an extremal graph, thus proving the lower bound.

\begin{lemma}\label{lem:cographinductionextremal}
Let $i,j\geq k+2$. Then, $R_{k}^{\mathcal{CO}}(i,j)\geq (i-1)+R_{k}^{\mathcal{CO}}(i,j-k-1)$.
\end{lemma}

\begin{proof}
Let $H$ be a cograph on $R_{k}^{\mathcal{CO}}(i,j-k-1)-1$ vertices which has no $k$-dense $i$-set and no $k$-sparse $(j-k-1)$-set. Let also $T$ be the graph consisting of the join of a clique $K$ of size $i-k-2$ and an independent set $J$ of size $k+1$. We call $U$ the disjoint union of $T$ and $H$. Note that $U$ is a cograph on $(i-k-2)+(k+1)+R_{k}^{\mathcal{CO}}(i,j-k-1)-1=(i-2)+R_{k}^{\mathcal{CO}}(i,j-k-1)$ vertices. Observe that we have $\alpha_k(T)=k+1$ since any set in $T$ of size at least $k+2$ has a vertex from $K$. Hence, $\alpha_k(U)=\alpha_k(H)+\alpha_k(T)\leq (j-k-2)+(k+1)=j-1$. On the other hand, take a set $I$ in $U$ with $|I|=i$. Let $|I\cap J|=a$, $|I\cap H|=b$ and $|I\cap K|=i-a-b$. Firstly, we have $a+b\geq k+2$ since $|K|=i-k-2$. Now, if $a\geq 1$, any vertex $v \in I\cap J$ misses $a-1$ vertices from $I\cap J$ and $b$ vertices from $I\cap H$, so $v$ misses in total $a+b-1\geq k+1$ vertices in $I$. Therefore, $a\geq 1$ implies $I$ is not $k$-dense. If $a=0$, then we have $b\geq k+2$. Moreover, since $H$ has no $k$-dense $i$-set, we get $i-a-b\geq 1$. Now, any vertex $w\in I\cap K$ misses $b\geq k+2$ vertices in $I$, so $I$ is not $k$-dense. As a result, $U$ has no $k$-dense $i$-set and no $k$-sparse $j$-set, which implies $R_{k}^{\mathcal{CO}}(i,j)\geq (i-1)+R_{k}^{\mathcal{CO}}(i,j-k-1)$. \qed
\end{proof}

We are now ready to present the main result of this section, completing all defective Ramsey numbers in cographs.

\begin{theorem}\label{thm:cographkdefective}
	Let $i,j\geq k+2$. Then,	$$\displaystyle{R_k^{\mathcal{CO}}(i,j)=1+\dfrac{(i-1)(j-1)-\{i-1\}\{j-1\}}{k+1}}$$ where $\{x\}$ denotes the value of the integer $x$ modulo $k+1$. 
\end{theorem}

\begin{proof}
We will prove by strong induction on $i+j$. Firstly, the lemma holds if $i=k+2$ or $j=k+2$ from Remark \ref{rem:cactuskplustwo}. Then, let us assume the lemma holds for $i,j\geq k+2$ and $i+j<t$ for some $t\geq 2k+6$, and take $i,j\geq k+3$ with $i+j=t$. Consider a cograph $G$ on $1+\dfrac{(i-1)(j-1)-\{i-1\}\{j-1\}}{k+1}$ vertices. We will prove that $G$ has either a $k$-dense $i$-set or a $k$-sparse $j$-set.

 Suppose $G$ is disconnected, take a connected component $C$ of $G$. If at least one of $C$ or $G-C$ has a $k$-dense $i$-set, then we are done. Thus, assume there is no $k$-dense $i$-set in $C$ nor in $G-C$. We note that $\alpha_k(G-C)=\alpha_k(G)-\alpha_k(C)$ by Remark \ref{rem:ksparseunion}. It follows that $|C|\leq R_{k}^{\mathcal{CO}}(i,\alpha_k(C)+1)-1$ and $|G-C|\leq R_{k}^{\mathcal{CO}}(i,\alpha_k(G)-\alpha_k(C)+1)-1$. Since $i\geq k+2$, by using induction hypothesis and Remark \ref{rem:generalsmallvaluesofiandj}, we have the following for any integer $1\leq x\leq j-1$.
  $$R_{k}^{\mathcal{CO}}(i,x)=\begin{cases}
1+\dfrac{(i-1)(x-1)-\{i-1\}\{x-1\}}{k+1}, & \text{ if } x\geq k+2 \\
x, & \text{ if } x\leq k+1.
\end{cases}$$ 
Since $i\geq k+2$, by Lemma \ref{lem:cographmoduloremainder2}, we have $R_{k}^{\mathcal{CO}}(i,x) \leq 1+\dfrac{(i-1)(x-1)-\{i-1\}\{x-1\}}{k+1}$ for any integer $x$ such that $1\leq x \leq j-1$. This implies $|C|\leq \dfrac{(i-1)\alpha_k(C)-\{i-1\}\{\alpha_k(C)\}}{k+1}$. Now, if $\alpha_k(G)\geq j$, we are done and so we can assume $\alpha_k(G)\leq j-1$. Then we have $\alpha_k(G-C)\leq j-1-\alpha_k(C)$. It follows that $|G-C|\leq R_{k}^{\mathcal{CO}}(i,j-\alpha_k(C))-1$. Then, since $i\geq k+2$, we conclude $|G-C|\leq \dfrac{(i-1)(j-1-\alpha_k(C))-\{i-1\}\{j-1-\alpha_k(C)\}}{k+1}$. Consequently, we have 
\begin{eqnarray*}
	|G|&=&|C|+|G-C|\\
	&\leq& \dfrac{(i-1)\alpha_k(C)-\{i-1\}\{\alpha_k(C)\}+(i-1)(j-1-\alpha_k(C))-\{i-1\}\{j-1-\alpha_k(C)\}}{k+1}\\
	&=& \dfrac{(i-1)(j-1)-\{i-1\}\big(\{j-1-\alpha_k(C)\}+\{\alpha_k(C)\}\big)}{k+1}.
\end{eqnarray*}
Now, since $|G|=1+\dfrac{(i-1)(j-1)-\{i-1\}\{j-1\}}{k+1}$, we see $$1+\dfrac{(i-1)(j-1)-\{i-1\}\{j-1\}}{k+1}\leq \dfrac{(i-1)(j-1)-\{i-1\}\big(\{j-1-\alpha_k(C)\}+\{\alpha_k(C)\}\big)}{k+1}$$
This implies $1\leq \dfrac{\{i-1\}\cdot\Big(\{j-1\}-\{\alpha_k(C)\}-\{j-1-\alpha_k(C)\}\Big)}{k+1}$. However, from Lemma \ref{lem:cographmoduloremainder}, we know $\{j-1\}-\{\alpha_k(C)\}\leq \{j-1-\alpha_k(C)\}$ and so $$1\leq \dfrac{\{i-1\}\cdot\Big(\{j-1\}-\{\alpha_k(C)\}-\{j-1-\alpha_k(C)\}\Big)}{k+1}\leq 0,$$which is a contradiction. As a result, we have $\alpha_k(G)\geq j$ and we are done.

So, suppose $G$ is connected. Since $G$ is a cograph, we know that $\overline{G}$ is a disconnected cograph on $1+\dfrac{(i-1)(j-1)-\{i-1\}\{j-1\}}{k+1}=1+\dfrac{(j-1)(i-1)-\{j-1\}\{i-1\}}{k+1}$ vertices. In the previous case we have proved that $\overline{G}$ has either a $k$-dense $j$-set or a $k$-sparse $i$-set, which implies $G$ has either a $k$-dense $i$-set or a $k$-sparse $j$-set.  This completes the proof of the upper bound.

\noindent Now, we need to show that $R_{k}^{\mathcal{CO}}(i,j)\geq 1+\dfrac{(i-1)(j-1)-\{i-1\}\{j-1\}}{k+1}$. If both of $i$ and $j$ are less than $2k+3$, then we reach our desired conclusion by Lemma \ref{lem:cographsmalliandj}. Therefore, assume $\max\{i,j\}\geq 2k+3$.

 If $j\geq 2k+3$, then by noting $\{j-1\}=\{j-k-2\}$, we have $R_{k}^{\mathcal{CO}}(i,j-k-1)=1+\dfrac{(i-1)(j-k-2)-\{i-1\}\{j-1\}}{k+1}$ from the induction hypothesis. Thus, from Lemma \ref{lem:cographinductionextremal}, we get $R_{k}^{\mathcal{CO}}(i,j)\geq (i-1)+ R_{k}^{\mathcal{CO}}(i,j-k-1) =1+\dfrac{(i-1)(j-1)-\{i-1\}\{j-1\}}{k+1}$.

 If $i\geq 2k+3$, then by noting $\{i-1\}=\{i-k-2\}$, we have $R_{k}^{\mathcal{CO}}(j,i-k-1)=1+\dfrac{(j-1)(i-k-2)-\{j-1\}\{i-1\}}{k+1}$ from induction hypothesis. Thus, from Lemma \ref{lem:cographinductionextremal}, we get $R_{k}^{\mathcal{CO}}(j,i)\geq (j-1)+ R_{k}^{\mathcal{CO}}(j,i-k-1) =1+\dfrac{(j-1)(i-1)-\{j-1\}\{i-1\}}{k+1}$. Since $R_{k}^{\mathcal{CO}}(i,j)=R_{k}^{\mathcal{CO}}(j,i)$, we are done. \qed

\end{proof}


\section{Defective versus classical Ramsey numbers}\label{sec:conj}

To compare defective Ramsey numbers with the classical Ramsey numbers, the authors in \cite{defectiveRamseyJohnChappell} conjectured that $R_k(k+i,k+j)-k\leq R(i,j)$ holds for all $i,j,k\geq 0$. We will examine this conjecture when restricted to graph classes studied in this paper and show that it holds in forests, cacti and cographs whereas it fails in bipartite graphs and split graphs.

\begin{proposition}
The inequality $R^{\mathcal G}_k(k+i,k+j)-k\leq R^{\mathcal G}(i,j)$ holds if ${\mathcal G}$ is the class of i) forests, ii) cacti or iii) cographs, and does not hold if ${\mathcal G}$ is the class of iv) bipartite graphs or v) split graphs.
\end{proposition}

\begin{proof}

i) For forests, we have $R_{k}^{\mathcal{FO}}(i,j)=j+\Big\lfloor\dfrac{j-1}{k+1}\Big\rfloor$ for $i\geq k+3$ and $j\geq k+2$ from Theorem \ref{thm:forest}. Thus, we get $R_{k}^{\mathcal{FO}}(k+i,k+j)-k=j+\Big\lfloor\dfrac{k+j-1}{k+1}\Big\rfloor$. Observe that $k\geq 1$ and $j\geq 3$ imply 
\begin{eqnarray*}
	j+\Big\lfloor\dfrac{k+j-1}{k+1}\Big\rfloor &\leq& j+1+\Big\lfloor\dfrac{j-1}{k+1}\Big\rfloor \\
	&\leq& j+1+\dfrac{j-1}{2}\\
	&=&\dfrac{3j+1}{2}\leq 2j-1 = R^{\mathcal{FO}}(i,j)
\end{eqnarray*}
where the last equality comes from Theorem \ref{thm:forestgeneralRamsey}.

ii) In cacti, we have $R_{k}^{\mathcal{CA}}(i,j)=\Big\lceil\dfrac{(k+1)(j-1)}{k}\Big\rceil$ for $i\geq k+4$ and $j,k\geq 2$ from Theorem \ref{thm:cactus}. Thus, by using $k\geq 2$, we get $$R_{k}^{\mathcal{CA}}(k+i,k+j)-k=1+\Big\lceil\dfrac{(k+1)(j-1)}{k}\Big\rceil\leq 1+3(j-1)=R^{\mathcal{CA}}(i,j)$$ from Theorem \ref{thm:cactusgeneralRamsey}. 

iii) For cographs, recall from Theorem \ref{thm:cographkdefective} that $R_{k}^{\mathcal{CO}}(i,j) = 1+\dfrac{(i-1)(j-1)-\{i-1\}\{j-1\}}{k+1}$ for $i,j\geq k+2$ with the notation used in Section \ref{sec:cographs}. Also, we have $\{x\}\geq 0$ by definition, and $i,j\geq 2$ implies $(i-2)(j-2)\geq 0$ and so $(i-1)(j-1)\geq i+j-3$. Thus,

\begin{eqnarray*}
R_{k}^{\mathcal{CO}}(k+i,k+j)-k&=& 1+\dfrac{(i+k-1)(j+k-1)-\{i+k-1\}\{j+k-1\}}{k+1}-k\\
&\leq& 1+\dfrac{(i+k-1)(j+k-1)}{k+1}-k\\
&=&1+\dfrac{k^2+k(i-1)+k(j-1)}{k+1}+\dfrac{(i-1)(j-1)}{k+1}-k\\
&=&1+\dfrac{k(i+j-3)}{k+1}+\dfrac{(i-1)(j-1)}{k+1}\\
&\leq&1+\dfrac{k(i-1)(j-1)}{k+1}+\dfrac{(i-1)(j-1)}{k+1}\\
&=&1+(i-1)(j-1)= R^{\mathcal{CO}}(i,j)
\end{eqnarray*}
\noindent where the last equality comes from Theorem \ref{thm:cographgeneralRamsey}. 

iv) In bipartite graphs, we have $R_{1}^{\mathcal{BIP}}(1+i,1+j)-1=2j=1+R^{\mathcal{BIP}}(i,j)$ from Theorem \ref{thm:bipartite1defective} and Theorem \ref{thm:bipartitegeneralRamsey} for $i\geq 4$ and $j\geq 20$.

v) In split graphs, we have $R_{k}^{\mathcal{SP}}(i,j)=i+j-1=R^{\mathcal{SP}}(i,j)$ for $i,j\geq 2k+3$ from Corollary \ref{cor:splitsufficientlylarge}. Then, we conclude $R_{k}^{\mathcal{SP}}(k+i,k+j)-k=R_{k}^{\mathcal{SP}}(i,j)+k=R^{\mathcal{SP}}(i,j)+k>R^{\mathcal{SP}}(i,j)$ for all $i,j\geq 2k+3$.
\end{proof}

\section{Conclusion}\label{sec:conclusion}

In this paper, we considered the computation of defective Ramsey numbers in various graph classes, namely forests ($\mathcal{FO}$), cacti ($\mathcal{CA}$), bipartite graphs ($\mathcal{BIP}$), split graphs ($\mathcal{SP}$) and cographs ($\mathcal{CO}$). Obtained results, conjectures and open questions mentioned in previous sections are summarized in Table \ref{table:summary}. 



\begin{table}[ht]
\large
\scalebox{0.5}{
		\begin{tabular}{||D||C||C||C||C||C||}
			\hline\hline
			 & \text{Conditions on $i$ and $j$}
			& $$k=1$$ & $$k=2$$ & $$k=3$$ & $$k\geq4$$ \\ 
			\hline
			$$R_k^{\mathcal{FO}}(i,j)$$ & \text{for all $i$ and $j$} &\multicolumn{4}{|C||}{j+\Big\lfloor \dfrac{j-1}{k+1}\Big\rfloor} \\ [-5 ex] 
			\hline
			
			\multirow{3}{*}{$R_k^{\mathcal{CA}}(i,j)$} & i=k+3 \text{ and } j\leq 2k+1 & \multicolumn{4}{|C||}{j-1+\Big\lceil\dfrac{j-1}{k}\Big\rceil} \\
			\cline{2-3} \cline{4-6}
			& i=k+3 \text{ and } j\geq 2k+2  & \multicolumn{3}{|C||}{j-1+\Big\lceil\dfrac{j-1}{k}\Big\rceil} & \text{OPEN} \\ [1.2ex]
			\cline{2-6}
			& i\geq k+4 & 
			\begin{cases} 2j-2 & \text{if $j$ is even}\\
				2j-1 & \text{if $j$ is odd}\end{cases} & \multicolumn{3}{|C||}{{j-1+\Big\lceil\dfrac{j-1}{k}\Big\rceil}} \\ 
			\hline
			
			\multirow{5}{*}{$R_k^{\mathcal{BIP}}(i,j)$} & \shortstack[l]{\quad\,\,$i=4$ \\ $j \in \{4,5,6 \}$ } & 2j-3 & \multicolumn{3}{|C||}{\multirow{4}{*}{ \shortstack[l]{ For $i\geq 2k+3$, we have \\
			$R_{k}^{\mathcal{BIP}}(i,j)=\begin{cases}
			2j-1-k, & \text{if } k+2\leq j\leq 2k.\\
			2j-1, & \text{if } j\geq 2k+1.
			\end{cases}$ \\ \\ \\ \\ 
			\vspace{1cm}
			\\
			\text{OPEN for} $2k+2\geq i\geq k+3$ \text{and} $j\geq k+2$.   } }} \\
			\cline{2-3}
			& \shortstack[l]{\,\,\,\,$i=4$ \\ $j \in \{3,7\}$ } & 2j-2 & \multicolumn{3}{|C||}{} \\
			\cline{2-3}
			& \shortstack[l]{\qquad\quad\,\,$i=4$ \\ $j \in \{10,11,12,18,19 \}$ } & \textbf{CONJ: }2j-1 & \multicolumn{3}{|C||}{}   \\
			\cline{2-3}
			&\shortstack[l]{\qquad\qquad\,\,$i=4$ \\ $j \in \{8,9,13,14,15,16,17 \}$ } & 2j-1 & \multicolumn{3}{|C||}{}   \\
			\cline{2-3}
			& i=4,\,j\geq 20 \text{ or } i\geq 5  & 2j-1 & \multicolumn{3}{|C||}{}   \\
			\hline
			
			\multirow{3}{*}{$R_k^{\mathcal{SP}}(i,j)$} & (i-k-2)(j-k-2)\geq (k+1)^2 & \multicolumn{4}{|C||}{i+j-1} \\ 
			\cline{2-6}
			& \shortstack[l]{$(i-k-2)(j-k-2)<(k+1)^2$ \\ \qquad \qquad \text{ and } $i=j$ }  & \multicolumn{4}{|C||}{3i-2k-4} \\
			\cline{2-6}
			& \shortstack[l]{$(i-k-2)(j-k-2)<(k+1)^2$ \\ \qquad \qquad \text{ and } $i\neq j$ } &\multicolumn{2}{|C||}{i+j-1-\Big\lceil\dfrac{(k+1)^2-(i-k-2)(j-k-2)}{\min\{i,j\}}\Big\rceil}  &\multicolumn{2}{C||}{\textbf{CONJ: } i+j-1-\Big\lceil\dfrac{(k+1)^2-(i-k-2)(j-k-2)}{\min\{i,j\}}\Big\rceil  } \\ 
			\hline
			
			$$R_k^{\mathcal{CO}}(i,j)$$ &\text{for all $i$ and $j$ } &\multicolumn{4}{|C||}{1+\dfrac{(i-1)(j-1)-\{i-1\}\{j-1\}}{k+1}} \\  
			\hline \hline
		\end{tabular} }
\caption{Summary of the results obtained in this paper and open questions.}\label{table:summary}
\end{table}

\noindent Apart from the conjectures and open questions formulated in this paper, one can study other graph classes from the same perspective. Whenever we are not likely to obtain a general formula in some graph class, one can also address the computation of small defective Ramsey numbers using efficient enumeration algorithms. Such a study has been initiated in \cite{1defectiveperfectTinazOylumJohn} for perfect graphs (denoted by $\mathcal{PG}$) for the computation of $R_1^{\mathcal{PG}}(5,5)$; further defective Ramsey numbers in $\mathcal{PG}$ for $k\geq 2$ can be considered in the same manner. Bipartite graphs and chordal graphs are other candidate graph classes for which efficient enumeration algorithms are likely to provide some small defective Ramsey numbers.


\end{document}